\documentclass[11pt]{amsart}
\usepackage{amssymb,amsmath,hyperref,color,enumitem,caption,comment}
\usepackage{fullpage}
\usepackage[noadjust]{cite}
\usepackage[abs]{overpic}
\usepackage{cleveref}
\usepackage{tikz}
\usepackage{tikz-cd}
\usepackage{rotating}

\usetikzlibrary{calc}
\usetikzlibrary{arrows,decorations.markings}
\pdfoutput=1

\newtheorem{thm}{Theorem}[section]
\newtheorem*{mainirredthm}{Theorem~\ref{thm:main_irred}}

\newtheorem{cor}[thm]{Corollary}
\newtheorem{lem}[thm]{Lemma}
\newtheorem{prop}[thm]{Proposition}
\newtheorem{conj}[thm]{Conjecture}
\newtheorem{ques}[thm]{Question}

\theoremstyle{definition}
\newtheorem{ex}[thm]{Example}
\newtheorem{defn}[thm]{Definition}

\newtheorem{rem}[thm]{Remark}

\newcommand{\bbA}{\mathbb{A}}
\newcommand{\bbC}{\mathbb{C}}

\newcommand{\bbN}{\mathbb{N}}
\newcommand{\bbP}{\mathbb{P}}
\newcommand{\bbQ}{\mathbb{Q}}

\newcommand{\bbZ}{\mathbb{Z}}

\newcommand{\calB}{\mathcal{B}}
\newcommand{\calC}{\mathcal{C}}
\newcommand{\calD}{\mathcal{D}}

\newcommand{\calM}{\mathcal{M}}

\newcommand{\calO}{\mathcal{O}}

\newcommand{\calS}{\mathcal{S}}
\newcommand{\calT}{\mathcal{T}}

\newcommand{\OK}{\calO_K}

\newcommand{\Kbar}{\overline{K}}

\newcommand{\calDbar}{\overline{\calD}}
\newcommand{\calObar}{\overline{\calO}}

\newcommand{\QQbar}{\overline{\bbQ}}
\newcommand{\ZZbar}{\overline{\bbZ}}

\renewcommand{\hbar}{\overline{h}}

\newcommand{\frakp}{\mathfrak{p}}
\newcommand{\frakq}{\mathfrak{q}}

\newcommand{\Gal}{\operatorname{Gal}}

\newcommand{\Spec}{\operatorname{Spec}}

\newcommand{\PrePer}{\operatorname{PrePer}}

\newcommand{\disc}{\operatorname{disc}}
\newcommand{\Aut}{\operatorname{Aut}}

\newcommand{\dyn}{\operatorname{dyn}}
\newcommand{\Xdyn}{X^{\dyn}}

\newcommand{\Ell}{\operatorname{ell}}
\newcommand{\Xell}{X^{\Ell}}

\newcommand{\f}{\boldsymbol{f}}
\newcommand{\G}{\mathbf{G}}
\renewcommand{\H}{\mathbf{H}}
\newcommand{\x}{\mathbf{x}}
\renewcommand{\P}{\boldsymbol{P}}
\newcommand{\Q}{\boldsymbol{Q}}
\newcommand{\T}{\mathbf{T}}

\newcommand{\balpha}{\boldsymbol{\alpha}}

\renewcommand{\tilde}{\widetilde}
\renewcommand{\epsilon}{\varepsilon}
\renewcommand{\phi}{\varphi}

\renewcommand{\labelenumi}{(\Alph{enumi})}

\usepackage{verbatim}

\allowdisplaybreaks

\numberwithin{equation}{section}

\title[Dynamical Modular Curves]{Dynamical Modular Curves for Quadratic Polynomial Maps}
\author{John R. Doyle}
\address{Department of Mathematics \\
University of Rochester \\
Rochester, NY 14627}
\curraddr{Mathematics \& Statistics Department \\
Louisiana Tech University \\
Ruston, LA 71272}
\email{jdoyle@latech.edu}

\begin{document}

\begin{abstract}
Motivated by the dynamical uniform boundedness conjecture of Morton and Silverman, specifically in the case of quadratic polynomials, we give a formal construction of a certain class of dynamical analogues of classical modular curves. The preperiodic points for a quadratic polynomial map may be endowed with the structure of a directed graph satisfying certain strict conditions; we call such a graph {\it admissible}. Given an admissible graph $G$, we construct a curve $X_1(G)$ whose points parametrize quadratic polynomial maps --- which, up to equivalence, form a one-parameter family --- together with a collection of marked preperiodic points that form a graph isomorphic to $G$. Building on work of Bousch and Morton, we show that these curves are irreducible in characteristic zero, and we give an application of irreducibility in the setting of number fields. We end with a discussion of the Galois theory associated to the preperiodic points of quadratic polynomials, including a certain Galois representation that arises naturally in this setting.
\end{abstract}

\maketitle

\section{Introduction}\label{sec:intro}

Let $K$ be a field, and let $\phi(x) \in K(x)$ be a rational function of degree $d \ge 2$, viewed as a self-map of $\bbP^1(K)$. For an integer $n \ge 0$, we denote by $\phi^n$ the $n$-fold composition of $\phi$; that is, $\phi^0$  is the identity map, and $\phi^n = \phi \circ \phi^{n-1}$ for each $n \ge 1$. Given $\alpha \in K$, the \textbf{orbit} of $\alpha$ under $\phi$ is the set
	\[
		\calO_\phi(\alpha) := \{\phi^n(\alpha) : n \ge 0\}.
	\]
We say that $\alpha \in K$ is \textbf{preperiodic} for $\phi$ if $|\calO_\phi(\alpha)| < \infty$. In this case, there exist integers $M \ge 0$ and $N \ge 1$ such that $\phi^{M+N}(\alpha) = \phi^M(\alpha)$. If $M$ and $N$ are minimal with this property, we call $M$ the \textbf{preperiod} and $N$ the \textbf{eventual period} of $\alpha$ under $\phi$, and we refer to the pair $(M,N)$ as the \textbf{preperiodic portrait} (or simply \textbf{portrait}) of $\alpha$. We further define $\alpha$ to be \textbf{periodic} (resp., \textbf{strictly preperiodic}) if the preperiod of $\alpha$ is zero (resp., positive). We denote by $\PrePer(\phi,K)$ the set of $K$-rational preperiodic points for $\phi$. The set $\PrePer(\phi,K)$ comes naturally equipped with the structure of a directed graph, with an edge $\alpha \to \beta$ if and only if $\phi(\alpha) = \beta$, and we denote this graph by $G(\phi,K)$. We illustrate this with an example in Figure~\ref{fig:graph_ex}.

\begin{rem}
The point at infinity is a fixed point for every polynomial map. Following the convention of \cite{poonen:1998} (also used in \cite{doyle:quads,doyle/faber/krumm:2014}), we omit the vertex $\infty$ from the graph $G(\phi,K)$ when $\phi$ is a polynomial, as in Figure~\ref{fig:graph_ex}.
\end{rem}

One of the major motivations in the area of arithmetic dynamics is the analogy between preperiodic points for rational maps and torsion points on abelian varieties. Northcott \cite{northcott:1950} showed that if $\phi$ is a rational map of degree at least 2 defined over a number field $K$, then $\PrePer(\phi,\Kbar)$ is a set of bounded height, hence $\PrePer(\phi,K)$ is finite. Motivated by the strong uniform boundedness conjecture for elliptic curves --- known at the time for low degree number fields by the work of Mazur, Kamienny, and Abramovich \cite{mazur:1977,kamienny:1992,kamienny/mazur:1995,abramovich:1995}, and subsequently proven in full generality by Merel \cite{merel:1996} --- Morton and Silverman have suggested the following dynamical analogue:

\begin{conj}[{\cite[p. 100]{morton/silverman:1994}}]\label{conj:ubc}
Let $n \ge 1$ and $d \ge 2$. There exists a constant $B(n,d)$ such that if $[K:\bbQ] = n$ and $\phi(x) \in K(x)$ has degree $d$, then $|\PrePer(\phi,K)| \le B(n,d)$.
\end{conj}

\begin{figure}
\centering
%    \begin{tikzpicture}[scale=.6]
%\tikzset{vertex/.style = {}}
%\tikzset{every loop/.style={min distance=10mm,in=45,out=-45,->}}
%\tikzset{edge/.style={decoration={markings,mark=at position 1 with %
%    {\arrow[scale=1.5,>=stealth]{>}}},postaction={decorate}}}
%% vertices
%\node[vertex, label={left:$-\frac{7}{4}$}] (3a) at  (0,0) {$\bullet$};
%\node[vertex, label={left:$\frac{5}{4}$}] (3b) at  (0,2) {$\bullet$};
%\node[vertex, label={above:$-\frac{1}{4}$}] (3c) at  (1.7,1) {$\bullet$};
%%
%\node[vertex, label={left:$\frac{1}{4}$}] (13a) at  (-1,-1.7) {$\bullet$};
%\node[vertex, label={left:$\frac{7}{4}$}] (13b) at (-1,3.7) {$\bullet$};
%\node[vertex, label={above:$-\frac{5}{4}$}] (13c) at (3.7, 1) {$\bullet$};
%%
%\node[vertex, label={right:$\frac{3}{4}$}] (23a) at (5.4, 2) {$\bullet$};
%\node[vertex, label={right:$-\frac{3}{4}$}] (23b) at (5.4, 0) {$\bullet$};
%% edges
%\draw[edge] (3a) to (3b);
%\draw[edge] (3b) to (3c);
%\draw[edge] (3c) to (3a);
%\draw[edge] (13a) to (3a);
%\draw[edge] (13b) to (3b);
%\draw[edge] (13c) to (3c);
%\draw[edge] (23a) to (13c);
%\draw[edge] (23b) to (13c);
%\end{tikzpicture}
	\includegraphics{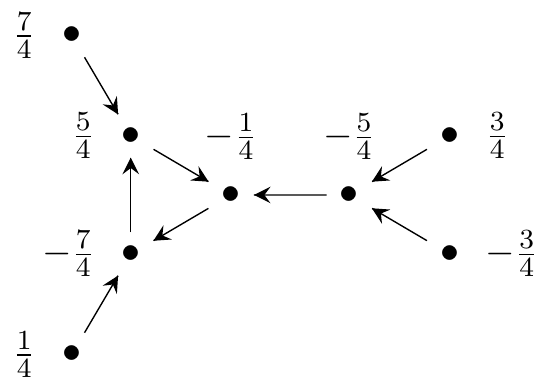}
	\caption{The graph $G(z^2 - 29/16, \bbQ)$}
	\label{fig:graph_ex}
\end{figure}

Conjecture~\ref{conj:ubc} is not currently known to hold for any $n \ge 1$ and $d \ge 2$, even if we further restrict our attention to \emph{polynomials} of degree $d$. The simplest polynomial case --- namely, quadratic polynomials over $\bbQ$, corresponding to the case $(n,d) = (1,2)$ --- is still far from being solved, though there has been significant progress in this direction. Before describing what is currently known, we introduce some notation.

Two polynomial maps $\phi, \psi \in K[x]$ are {\bf linearly equivalent} if there exists a polynomial $\ell(x) = ax + b$ with $a,b \in K$ and $a \ne 0$ such that $\psi = \ell^{-1} \circ \phi \circ \ell$. This is the appropriate notion of equivalence in a dynamical setting since conjugation commutes with iteration; in particular, $\ell$ induces a directed graph isomorphism $G(\psi,K) \to G(\phi,K)$. Every quadratic polynomial defined over a number field $K$ is equivalent to a unique polynomial of the form $f_c(x) := x^2 + c$ with $c \in K$. Therefore, when discussing Conjecture~\ref{conj:ubc} for quadratic polynomials, it suffices to consider only maps of the form $f_c$. For reference, we restate the Morton-Silverman conjecture for quadratic polynomials.

\begin{conj}\label{conj:ubc_quads}
Let $n \in \bbN$. There exists a constant $B(n)$ such that if $[K:\bbQ] = n$ and $c \in K$, then $|\PrePer(f_c,K)| \le B(n)$.
\end{conj}

The best known result toward Conjecture~\ref{conj:ubc_quads} in the case $n = 1$ is the following theorem of Poonen, which may be considered a dynamical analogue of Mazur's classification theorem \cite{mazur:1977} for rational torsion subgroups of elliptic curves over $\bbQ$:

\begin{thm}[{Poonen \cite[Cor. 1]{poonen:1998}}]\label{thm:poonen}
Let $c \in \bbQ$. If $f_c$ does not admit rational points of period greater than $3$, then $|\PrePer(f_c,\bbQ)| \le 9$. Moreover, under these assumptions, $G(f_c,\bbQ)$ is isomorphic to one of the twelve directed graphs appearing in \cite[Fig. 1]{poonen:1998}
\end{thm}

It was conjectured in \cite{flynn/poonen/schaefer:1997} that if $c \in \bbQ$, then $f_c$ has no rational points of period greater than three; if this holds, then Poonen's result would give the first complete case of Conjecture~\ref{conj:ubc_quads}. Thus far, it has been shown that $f_c$ cannot admit rational points of period four (Morton \cite{morton:1998}), period five (Flynn-Poonen-Schaefer \cite{flynn/poonen/schaefer:1997}), or --- assuming certain standard conjectures on $L$-series of curves --- period six (Stoll \cite{stoll:2008}). Moreover, Hutz and Ingram \cite{hutz/ingram:2013} have gathered a significant amount of experimental data to support non-existence of rational points of period greater than three.

Some progress toward Conjecture~\ref{conj:ubc_quads} has also been made in the case $n = 2$. Computational evidence in \cite{hutz/ingram:2013, doyle/faber/krumm:2014} suggests the following:

\begin{conj}\label{conj:ubc2}
Let $K$ be a quadratic field, and let $c \in K$. Then $G(f_c,K)$ is isomorphic to one of the $46$ directed graphs appearing in \cite[App. B]{doyle/faber/krumm:2014}. In particular, $|\PrePer(f_c,K)| \le 15$.
\end{conj}

In addition to the computational evidence referenced above, there has been considerable progress made in the direction of Conjecture~\ref{conj:ubc2} in recent \cite{doyle/faber/krumm:2014, doyle:quads} and ongoing \cite{doyle/krumm/wetherell} work.

Just as the theory of modular curves played a crucial role in the proof of the uniform boundedness theorem for elliptic curves, one may define \emph{dynamical} modular curves in order to work toward Conjecture~\ref{conj:ubc_quads}. The classical modular curve $\Xell_1(N)$ parametrizes elliptic curves together with a marked torsion point of order $N$; analogously, one defines a curve $X_1(N) = \Xdyn_1(N)$ which parametrizes quadratic polynomials together with a marked point of period $N$. The curves $X_1(N)$, which we define in \textsection\ref{sec:dynatomic}, have been studied extensively since the 1980's, beginning with the work of Douady and Hubbard \cite{douady/hubbard:1985} and continuing, for example, in \cite{bousch:1992,lau/schleicher:1994,buff/lei:2014,morton:1996}. The main results of \cite{morton:1998,flynn/poonen/schaefer:1997,stoll:2008} mentioned above involved finding all rational points on the curves $X_1(N)$ with $N \in \{4,5,6\}$.

However, for providing a classification as in Theorem~\ref{thm:poonen} or Conjecture~\ref{conj:ubc2}, one needs a more general notion of a dynamical modular curve which parametrizes quadratic polynomial maps with {\it several} marked preperiodic points. The natural structure to encode a collection of preperiodic points for a given map is a directed graph, and because we are interested in graphs arising from quadratic polynomial dynamics, our graphs must satisfy certain strict conditions. We specify these conditions in \textsection\ref{sec:admissible}, and we call a graph satisfying these conditions {\it admissible}. Given an admissible graph $G$, we define a curve $X_1(G)$ whose points parametrize maps $f_c$ together with a collection of preperiodic points for $f_c$ which {\it generate} a directed graph isomorphic to $G$. Though several such curves have been explicitly constructed and used in work toward Conjecture~\ref{conj:ubc_quads} (e.g., in \cite{poonen:1998,doyle/faber/krumm:2014,doyle:quads,doyle/krumm/wetherell,doyle:cyclotomic}), a formal treatment of these curves has not appeared in the literature. The purpose of this article is to give such a treatment.

\begin{rem}\label{rem:silverman}
As this article was being completed, Silverman \cite{silverman} released a preprint discussing (weighted) directed graphs, also referred to as {\it portraits}, associated more generally to the dynamics of morphisms on projective spaces of arbitrary dimension. A general theory of moduli spaces for dynamical systems together with such portraits will appear in \cite{doyle/silverman}.
\end{rem}

Our main result concerning dynamical modular curves is the following:

\begin{thm}\label{thm:main_irred}
Let $K$ be a field of characteristic zero. For any admissible graph $G$, the curve $X_1(G)$ is irreducible over $K$.
\end{thm}

In fact, the properties of preperiodic points used to prove Theorem~\ref{thm:main_irred} yield a stronger Galois-theoretic result. We state a version here and defer to Theorem~\ref{thm:rep_iso} for a more precise statement.

\begin{thm}\label{thm:main_galois}
Let $f_t(z) := z^2 + t \in \bbC(t)[z]$, and let $L$ be the extension of $\bbC(t)$ generated by all preperiodic points of $f_t$. Then $\Gal(L/\bbC(t))$ consists of all permutations of the preperiodic points of $f_t$ that commute with $f_t$.
\end{thm}

This extends a result of Morton \cite{morton:1998gal}, who proved the analogous result for periodic points; see the discussion following the proof of Theorem~\ref{thm:main_irred} in \textsection\ref{sec:galois}.

We now give a brief outline of the article. In \textsection\ref{sec:dmc}, we formally define the dynamical modular curves $X_1(G)$ associated to admissible graphs $G$. Section~\ref{sec:dyn_mod} includes various properties of the curves $X_1(G)$, culminating in the proof of our main result (Theorem~\ref{thm:main_irred}). We then give an application of Theorem~\ref{thm:main_irred} in \textsection \ref{sec:realize}, and we end with a discussion of the Galois theory of preperiodic points, including a certain dynamical Galois representation, in \textsection\ref{sec:galois}.

\section*{Acknowledgments} This article originally began as the second chapter of my dissertation \cite{doyle:thesis} at the University of Georgia, though there have been many changes, improvements, and new results added since then. I would like to thank Bob Rumely, Pete Clark, Dino Lorenzini, and Robert Varley for several enlightening conversations as well as feedback on an early version of my thesis. I would like to thank Juan Rivera-Letelier for a number of helpful discussions, and I thank Joe Silverman and Tom Tucker for their comments and suggestions for improvements on an earlier draft.

\section{Dynamical modular curves}\label{sec:dmc}

\subsection{Dynatomic polynomials}\label{sec:dynatomic}

Let $N$ be a positive integer, and suppose that $\alpha,c \in K$ are such that $\alpha$ has period $N$ for $f_c(x)$. Then $(x,t) = (\alpha,c)$ is a solution to the equation $f_t^N(x) - x = 0$. However, this equation is also satisfied whenever $\alpha$ is periodic for $f_c$ with period a \emph{proper divisor} of $N$. One therefore defines the $N$th \textbf{dynatomic polynomial} to be the polynomial
	\[
		\Phi_N(x,t) := \prod_{n \mid N} \left(f_t^n(x) - x\right)^{\mu(N/n)} \in \bbZ[x,t],
	\]
where $\mu$ denotes the M\"obius function. That $\Phi_N$ is indeed a polynomial is shown in \cite[Thm. 4.5]{silverman:2007}, and the dynatomic polynomials yield a natural factorization
	\begin{equation}\label{eq:Ncycle}
		f_t^N(x) - x = \prod_{n \mid N} \Phi_n(x,t)
	\end{equation}
for each $N \in \bbN$ --- see \cite[p. 571]{morton/vivaldi:1995}. If $\Phi_N(\alpha,c) = 0$, we say that $\alpha$ has \textbf{formal period} $N$ for $f_c$. Every point of exact period $N$ has formal period $N$, but in some cases a point of formal period $N$ may have exact period $n$ a proper divisor of $N$. For example, $-1/2$ is a root of $\Phi_2(x,-3/4) = x^2 + x + 1/4$, thus has formal period 2 for $f_{-3/4}$; however, $-1/2$ is actually a fixed point for $f_{-3/4}$.

For each $N \in \bbN$, set
	\begin{align*}
		D(N) &:= \deg_x \Phi_N(x,t) = \sum_{n \mid N} \mu(N/n)2^n ; \\
		R(N) &:= \frac{D(N)}{N}.
	\end{align*}
The number $D(N)$ (resp., $R(N)$) represents the number of points of period $N$ (resp., periodic cycles of length $N$) for a general quadratic polynomial $f_c$ over $\QQbar$, excluding the fixed point at infinity in the case $N = 1$. The first few values of $D(N)$ and $R(N)$ are shown in Table~\ref{tab:d_and_r}.

\begin{table}
\centering
\renewcommand{\arraystretch}{1.5}
\caption{Small values of $D(N)$ and $R(N)$}
\label{tab:d_and_r}
\begin{tabular}{|c||c|c|c|c|c|c|c|c|c|c|}
\hline
	$N$ & 1 & 2 & 3 & 4 & 5 & 6 & 7 & 8 & 9 & 10\\
\hline
	$D(N)$ & 2 & 2 & 6 & 12 & 30 & 54 & 126 & 240 & 504 & 990 \\
\hline
	$R(N)$ & 2 & 1 & 2 & 3 & 6 & 9 & 18 & 30 & 56 & 99\\
\hline
\end{tabular}
\end{table}

Let $Y_1(N)$ be the affine plane curve defined by $\Phi_N(x,t) = 0$, which was shown to be irreducible over $\bbC$ by Bousch \cite[\textsection 3, Thm. 1]{bousch:1992}. We define $U_1(N)$ to be the Zariski open subset of $Y_1(N)$ on which $\Phi_n(x,t) \ne 0$ for each proper divisor $n$ of $N$, so that $(\alpha,c)$ lies on $Y_1(N)$ (resp., $U_1(N)$) if and only if $\alpha$ has formal (resp., exact) period $N$ for $f_c$. Since the curves $Y_1(N)$ are irreducible and distinct for all $N \in \bbN$ (by comparing degrees, for example), $U_1(N)$ is nonempty; hence $Y_1(N)$ is the closure of $U_1(N)$ in $\bbA^2$. We denote by $X_1(N)$ the normalization of the projective closure of $Y_1(N)$. 

\subsection{Generalized dynatomic polynomials}\label{sec:gen_dynatomic}

More generally, suppose $\alpha$ has preperiodic portrait $(M,N)$ for $f_c$ for some $M \ge 0$ and $N \ge 1$. In this case, we have $f_c^{M+N}(\alpha) - f_c^M(\alpha) = 0$; however, this equation is satisfied whenever $\alpha$ has portrait $(m,n)$ for some $0 \le m \le M$ and $n \mid N$. Therefore, for a pair of positive integers $M,N$, we define the \textbf{generalized dynatomic polynomial}
	\[ \Phi_{M,N}(x,t) := \frac{\Phi_N(f_t^M(x), t)}{\Phi_N(f_t^{M-1}(x), t)} \in \bbZ[x,t], \]
and we extend this definition to $M = 0$ by setting $\Phi_{0,N} := \Phi_N$. That $\Phi_{M,N}$ is a polynomial is proven in \cite[Thm. 1]{hutz:2015}. The generalized dynatomic polynomials give a natural factorization
	\[
		f_t^{M+N}(x) - f_t^M(x) = \prod_{m=0}^M\prod_{n \mid N} \Phi_{m,n}(x,t)
	\]
for all $M \ge 0$ and $N \ge 1$. If $\Phi_{M,N}(\alpha,c) = 0$, we say that $\alpha$ has \textbf{formal (preperiodic) portrait} $(M,N)$ for $f_c$. Just as in the periodic case, every point of exact portrait $(M,N)$ has formal portrait $(M,N)$, but the converse is not true in general. We now provide an alternative description of $\Phi_{M,N}$ that will be useful for later arguments.

\begin{lem}\label{lem:phiMNalt}
For any integers $M,N \ge 1$, we have $\Phi_{M,N}(x,t) = \Phi_N(-f_t^{M-1}(x),t)$.
\end{lem}

A slightly more general version of this lemma is proven in \cite[Lem. 5.7]{gao:2016}, so we only briefly sketch the idea here. A point $\alpha$ has portrait $(M,N)$ for $f_c$ if and only if $f_c^{M-1}(\alpha)$ is strictly preperiodic and $f_c^M(\alpha)$ has period $N$. Since each periodic point has a unique periodic preimage, and since the preimages of a given point under $f_c$ differ by a factor of $-1$, this happens if and only if $-f_c^{M-1}(\alpha)$ is a (nonzero) point of period $N$.

It follows from Lemma~\ref{lem:phiMNalt} that $\deg_x \Phi_{M,N} = 2^{M-1}D(N)$ and $\deg_t \Phi_{M,N} = 2^{M-2}D(N)$ for all positive integers $M$ and $N$. Let $Y_1((M,N))$ be the affine plane curve\footnote{We use the rather bulky notation $Y_1((M,N))$ instead of $Y_1(M,N)$ because we reserve $Y_1(M,N)$ for curves parametrizing quadratic polynomial maps together with marked points of period $M$ and $N$.} defined by $\Phi_{M,N}(x,t) = 0$. That these curves are irreducible over $\bbC$ follows from the work of Bousch \cite[p. 67]{bousch:1992}. We define $U_1((M,N))$ to be the Zariski open subset of $Y_1((M,N))$ given by
	\begin{align}\label{eq:phiMNconditions}
		\begin{split}
		\Phi_{M,n}(z,t) \ne 0 &\mbox{ for all $n < N$ with $n \mid N$};\\
		\Phi_{m,N}(z,t) \ne 0 &\mbox{ for all $m < M$ (if $M > 0$)}.
		\end{split}
	\end{align}
Arguing as in \textsection \ref{sec:dynatomic}, $U_1((M,N))$ is nonempty, so $Y_1((M,N))$ is the closure of $U_1((M,N))$ in $\bbA^2$, and we denote by $X_1((M,N))$ the normalization of the projective closure of $Y_1((M,N))$.

Finally, we note that a point $(\alpha,c)$ lies on $Y_1((M,N))$ (resp., $U_1((M,N))$) if and only if $\alpha$ has formal portrait (resp., exact portrait) $(M,N)$ for $f_c$. Indeed, the statement for the formal portrait holds by definition, and the statement for the exact portrait is a consequence of the following:

\begin{lem}[{\cite[Lem. 3.1]{doyle:2016}}]
Let $c \in \bbC$, and suppose that $\alpha \in \bbC$ has formal portrait $(M,N)$ and exact portrait $(m,n)$ for $f_c$. Then either $m = M$ or $n = N$.
\end{lem}

\subsection{Admissible graphs}\label{sec:admissible}
The curves $Y_1(N)$ and $Y_1((M,N))$ from the preceding section parametrize quadratic polynomial maps $f_c$ together with a single marked point of a given preperiodic portrait for $f_c$. The purpose of this section is to extend these notions and define curves that parametrize maps $f_c$ together with a collection of marked points of various portraits for $f_c$.

As mentioned in \textsection \ref{sec:intro}, given a number field $K$ and an element $c \in K$, the set of $K$-rational preperiodic points for $f_c$ forms a finite directed graph $G(f_c,K)$. Not every directed graph may be obtained in this way; for example, for any graph of the form $G(f_c,K)$, each vertex must have precisely one edge leading away from it. With this in mind, we define the class of graphs that we will consider throughout the remainder of this article.

\begin{defn}\label{defn:admissible}
A finite directed graph $G$ is \textbf{admissible} if it satisfies the following properties:

\renewcommand{\labelenumi}{(\alph{enumi})}
\begin{enumerate}
\item Every vertex of $G$ has out-degree 1 and in-degree either 0 or 2.
\item For each $N \ge 2$, $G$ contains at most $R(N)$ $N$-cycles.
\end{enumerate}
We say that $G$ is \textbf{strongly admissible} if $G$ additionally satisfies the following:
\begin{enumerate}
\setcounter{enumi}{2}
\item The graph $G$ has either zero or two fixed points.
\end{enumerate}
\renewcommand{\labelenumi}{(\Alph{enumi})}
\end{defn}

It may be that $G(f_c,K)$ fails to be admissible for some number field $K$ and parameter $c \in K$. We say a map is {\bf post-critically finite (PCF)} if all of its critical points are preperiodic; in our case, a map $f_c$ is PCF if and only if $0$ is preperiodic for $f_c$, since $0$ is the only (finite) critical point of $f_c$. If $f_c$ is PCF, then $c = f_c(0)$ is preperiodic, and the only preimage of $c$ under $f_c$ is $0$. This implies that $G(f_c,K)$ fails condition (a) from Definition~\ref{defn:admissible}. We now show that this is the only way $G(f_c,K)$ may fail to be admissible.

\begin{lem}\label{lem:inadmissible}
Let $K$ be a number field, and let $c \in K$. The graph $G(f_c,K)$ is admissible if and only if $f_c$ is not PCF, and is strongly admissible if and only if $f_c$ is not PCF and $c \ne 1/4$.
\end{lem}

\begin{proof}
Let $\beta$ be a vertex of $G(f_c,K)$. It is immediate from the definition of $G(f_c,K)$ that there is precisely one edge leading away from $\beta$. Now suppose that $\beta$ has a $K$-rational preimage $\alpha$. Then $f_c^{-1}(\beta) = \{\pm \alpha\}$, hence $\beta$ has exactly two $K$-rational preimages if and only if $\alpha \ne 0$. Thus $G(f_c,K)$ fails condition (a) from Definition~\ref{defn:admissible} if and only if $0 \in \PrePer(f_c,K)$; that is, if and only if $f_c$ is PCF. Condition (b) is always satisfied by $G(f_c,K)$, since a quadratic polynomial has at most $R(N)$ cycles of length $N$ by definition, thus $G(f_c,K)$ is admissible if and only if $f_c$ is not PCF.

Since a quadratic polynomial has at most two fixed points, condition (c) fails when $G(f_c,K)$ has precisely one fixed point. This occurs if and only if the discriminant
	\[
		\disc_x \Phi_1(x,c) = \disc_x (x^2 - x + c) = 1 - 4c
	\]
vanishes; i.e., if and only if $c = 1/4$. We conclude that $G(f_c,K)$ fails to be strongly admissible if and only if $f_c$ is PCF or $c = 1/4$.
\end{proof}

\begin{rem}
If $K$ is a number field, then the graph $G(f_{1/4},K)$ is indeed admissible. This follows from the fact that if $f_c$ is PCF, then $c$ is an algebraic integer.
\end{rem}

\begin{cor}\label{cor:admissible}
Given a number field $K$, there are only finitely many $c \in K$ for which $G(f_c,K)$ is not strongly admissible.
\end{cor}

\begin{proof}
By Lemma~\ref{lem:inadmissible}, it suffices to show that the set\footnote{Elements of $\calS$ are usually called \emph{Gleason points} or \emph{Misiurewicz points}, according to whether $0$ is periodic or strictly preperiodic for $f_c$.} $\calS := \{c \in K : 0 \in \PrePer(f_c,K)\}$ is finite. The proof of this fact is elementary and may be found in \cite[Prop. 4.22]{silverman:2007}.
\end{proof}

By the proof of Lemma~\ref{lem:inadmissible}, if the graph $G(f_c,K)$ is inadmissible, then $G(f_c,K)$ satisfies the requirements of admissibility {\it except} that $f_c$ is ramified at exactly one vertex (namely, 0) of $G(f_c,K)$. This motivates the following terminology:
\begin{defn}\label{defn:ramified}
A directed graph $G$ is \textbf{nearly admissible} if $G$ satisfies the conditions of admissibility, except that precisely one vertex has in-degree one.
\end{defn}
In particular, we note that the graph $G(f_c,K)$ is admissible or nearly admissible for all number fields $K$ and $c \in K$.

\begin{rem}
For general rational maps, it is more natural to treat preperiodic graphs as {\it weighted} directed graphs, where each vertex is weighted according to the ramification of the map at that point. Such graphs are considered in \cite{silverman, doyle/silverman}. Under that convention, our {\it admissible} graphs would be those whose vertices all have weight one, and our {\it nearly admissible} graphs would be those with a single vertex of weight greater than one (and necessarily equal to two). However, for the purposes of the current article, it suffices to work with unweighted graphs and treat PCF maps separately.
\end{rem}

For an admissible graph $G$, let $\f$ be the map on the set of vertices of $G$ defined so that $\f(P) = Q$ if and only if there is an edge $P \to Q$. For each vertex $P$, let $-P$ denote the unique vertex different from $P$ with $\f(P) = \f(-P)$. (The notation is intended to mirror the action of $f_c$ on $G(f_c,K)$.) We define the \textbf{orbit} of a vertex $P$ to be the set of vertices
	\[
		\calO(P) := \{\f^k(P) \ : \ k \ge 0 \}. 
	\]
We define the eventual period, preperiod, and preperiodic portrait of a vertex $P$ just as we did for the preperiodic points of a rational map. If $P$ has in-degree zero, we call $P$ an \textbf{initial point}.

Ultimately, our goal is to construct for each admissible graph $G$ an algebraic curve $Y_1(G)$ whose $K$-rational points, roughly speaking, correspond to quadratic maps $f_c$ for which $G(f_c,K)$ contains a subgraph isomorphic to $G$. To construct these curves as efficiently as possible (say, with as few equations as possible), we require a minimal amount of information that is necessary to uniquely determine a given admissible graph $G$.

A key observation is that if $\alpha \in \PrePer(f_c,K)$, then every element of the form $\pm f_c^k(\alpha)$ with $k \ge 0$ is also in $\PrePer(f_c,K)$. If $G$ is an admissible graph and $P$ is a vertex of $G$, we define
	\[
		\calObar(P) :=
			\{\pm \f^k(P) \ : \ k \ge 0\}.
	\]
It is clear from the definition that $\calO(P) \subseteq \calObar(P)$. We also make the trivial but important observations that if $P \in \calO(P')$, then $\calO(P) \subseteq \calO(P')$, and if $P \in \calObar(P')$, then $\calObar(P) \subseteq \calObar(P')$.

The subgraph of $G$ with vertex set $\calObar(P)$ is the minimal admissible subgraph of $G$ containing $P$. It therefore makes sense to say that $P$ \emph{generates} this particular subgraph. More generally, we make the following definition:	

\begin{defn}
Let $G$ be an admissible graph, and let $\P := \{P_1,\ldots,P_n\}$ be a set of vertices of $G$. We say that $\P$ \textbf{generates} the graph $G$ if every vertex of $G$ is in $\calObar(P_i)$ for some $i \in \{1,\ldots,n\}$. Equivalently, $\P$ generates $G$ if there is no proper admissible subgraph of $G$ containing all of the vertices $P_1,\ldots,P_n$. We say that $\P$ is a \textbf{minimal generating set} if any other generating set has at least $n$ elements.
\end{defn}

\begin{rem}
A minimal generating set for an admissible graph is not unique. For example, the graph $G$ shown in Figure~\ref{fig:admissible} is minimally generated by a point $P$ of period 2 and a point $Q$ of portrait (3,3). However, for any $P' \in \{\pm P, \pm \f(P)\}$ and any $Q' \in \{\pm Q\}$, the set $\{P',Q'\}$ minimally generates $G$.
\end{rem}

\begin{figure}
\centering
%    \begin{tikzpicture}[scale=.65]
%\tikzset{vertex/.style = {}}
%\tikzset{every loop/.style={min distance=10mm,in=45,out=-45,->}}
%\tikzset{edge/.style={decoration={markings,mark=at position 1 with %
%    {\arrow[scale=1.5,>=stealth]{>}}},postaction={decorate}}}
%% vertices
%\node[vertex] (3a) at  (0,0) {$\bullet$};
%\node[vertex] (3b) at  (0,2) {$\bullet$};
%\node[vertex] (3c) at  (1.7,1) {$\bullet$};
%\node[vertex] (13a) at  (-1,-1.7) {$\bullet$};
%\node[vertex] (13b) at (-1,3.7) {$\bullet$};
%\node[vertex] (13c) at (3.7, 1) {$\bullet$};
%\node[vertex] (23a) at (5.4, 2) {$\bullet$};
%\node[vertex] (23b) at (5.4, 0) {$\bullet$};
%\node[vertex, label=$Q$] (33a) at (7.33, 2.5) {$\bullet$};
%\node[vertex] (33b) at (7.33, 1.5) {$\bullet$};
%%
%\node[vertex, label=$P$] (2a) at (3.33,-1.7) {$\bullet$};
%\node[vertex] (2b) at (5.33,-1.7) {$\bullet$};
%\node[vertex] (12a) at (1.33,-1.7) {$\bullet$};
%\node[vertex] (12b) at (7.33,-1.7) {$\bullet$};
%% edges
%\draw[edge] (3a) to (3b);
%\draw[edge] (3b) to (3c);
%\draw[edge] (3c) to (3a);
%\draw[edge] (13a) to (3a);
%\draw[edge] (13b) to (3b);
%\draw[edge] (13c) to (3c);
%\draw[edge] (23a) to (13c);
%\draw[edge] (23b) to (13c);
%\draw[edge] (33a) to (23a);
%\draw[edge] (33b) to (23a);
%%
%\draw[edge] (2a) to[bend left=30] (2b);
%\draw[edge] (2b) to[bend left=30] (2a);
%\draw[edge] (12a) to (2a);
%\draw[edge] (12b) to (2b);
%\end{tikzpicture}
	\includegraphics{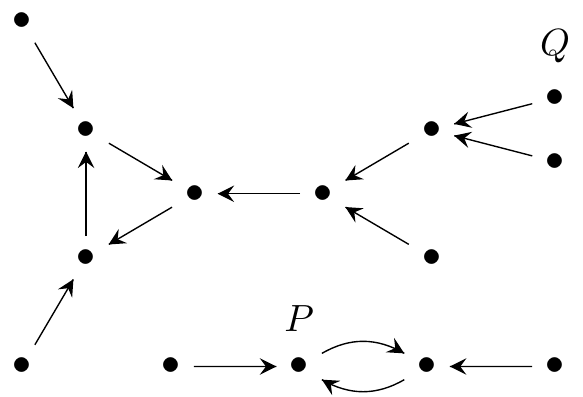}
	\caption{An admissible graph $G$ minimally generated by vertices $P$ and $Q$}
	\label{fig:admissible}
\end{figure}

\renewcommand{\labelenumi}{(\Alph{enumi})}

We now give two basic lemmas regarding minimal generating sets for admissible graphs.

\begin{lem}\label{lem:minimal_subset}
Let $G$ be an admissible graph, and let $\P := \{P_1,\ldots,P_n\}$ be a generating set for $G$. Then there exists a subset $\P' \subseteq \P$ that minimally generates $G$.
\end{lem}

\begin{proof}
Let $\Q := \{Q_1,\ldots,Q_m\}$ be any minimal generating set for $G$, and fix an index $i \in \{1,\ldots,m\}$. Since $\P$ generates $G$, we have $Q_i \in \calObar(P_{r(i)})$ for some $r(i) \in \{1,\ldots,n\}$, hence $\calObar(Q_i) \subseteq \calObar(P_{r(i)})$. Similarly, since $\Q$ generates $G$, we have $\calObar(P_{r(i)}) \subseteq \calObar(Q_j)$ for some $j \in \{1,\ldots,m\}$. Combining these two inclusions yields
	\begin{equation}\label{eq:gens}
		\calObar(Q_i) \subseteq \calObar(P_{r(i)}) \subseteq \calObar(Q_j).
	\end{equation}
Since $\Q$ is a minimal generating set for $G$, the inclusion $\calObar(Q_i) \subseteq \calObar(Q_j)$ implies that $i = j$. It then follows from \eqref{eq:gens} that $\calObar(Q_i) = \calObar(P_{r(i)})$. Therefore the vertex set of $G$ is equal to
	\[
		\bigcup_{i=1}^m \calObar(Q_i) = \bigcup_{i=1}^m \calObar(P_{r(i)}),
	\]
hence $\P' := \{P_{r(i)} : i \in \{1,\ldots,m\}\} \subseteq \P$ is a generating set for $G$. Since $|\P'| \le |\Q|$, $\P'$ is a minimal generating set.
\end{proof}

\begin{lem}\label{lem:more_generators}
Let $G$ and $H$ be admissible graphs with $H \subseteq G$, and let $\P$ and $\Q$ be minimal generating sets for $G$ and $H$, respectively. Then $|\P| \ge |\Q|$.
\end{lem}

\begin{proof}
Suppose to the contrary that $|\P| < |\Q|$. Every vertex of $H$ --- hence every element of $\Q$ --- lies in $\calObar(P)$ for some $P \in \P$. Since $|\P| < |\Q|$, there must be distinct elements $Q,Q' \in \Q$ and an element $P \in \P$ such that $Q,Q' \in \calObar(P)$. However, this implies that either $Q \in \calObar(Q')$ or $Q' \in \calObar(Q)$, contradicting minimality of $\Q$. Hence $|\P| \ge |\Q|$.
\end{proof}

Let $G$ be an admissible graph, and let $\P := \{P_1,\ldots,P_n\}$ be a (not necessarily minimal) generating set for $G$, and let $(M_i,N_i)$ be the portrait of $P_i$ for each $i \in \{1,\ldots,n\}$. Let $G_0 := \emptyset$ be the empty graph (i.e., the graph with no vertices), and for each $i \in \{1,\ldots,n\}$, we define the following:
\begin{equation}\label{eq:subgraphs}
	\begin{split}
		G_i &: \text{ the admissible subgraph of $G$ generated by $\{P_1,\ldots,P_i\}$};\\
		H_i &: \text{ the set of vertices in $G_i \setminus G_{i-1}$}.
	\end{split}
\end{equation}	
Note that $H_i \subseteq \calObar(P_i)$ for each $i$, and the vertex set of $G$ is the disjoint union of the $H_i$. If $\P$ is a \emph{minimal} generating set for $G$, then $\pm P_i \in H_i$ for each $i \in \{1,\ldots,n\}$; in particular, each $H_i$ is nonempty. We henceforth assume $\P$ is minimal.

Let $\calD = \calD(\P) \subseteq \{1,\ldots,n\}$ be the set of indices $i$ for which $\calO(P_i)$ is \emph{disjoint} from $G_{i-1}$; in this case, $H_i = \calObar(P_i)$. Set $\calDbar = \calDbar(\P) := \{1,\ldots,n\} \setminus \calD$. Note that $H_i$ contains a cycle if and only if $i \in \calD$.

If $i \in \calDbar$, then $\calO(P_i)$ intersects $G_{i-1}$, and we define the following:
\begin{equation}\label{eq:indices}
	\begin{split}
		\kappa_i &: \text{ the smallest integer $k$ for which $\f^k(P_i) \in G_{i-1}$};\\
		j_i &: \text{ the unique index $j$ for which $\f^{\kappa_i}(P_i) \in H_j$}.\\
		\lambda_i &: \text{ the smallest integer $\ell$ for which $\f^\ell(P_{j_i}) = -\f^{\kappa_i}(P_i)$}.
	\end{split}
\end{equation}
Note that $\kappa_i \ge 1$ since $P_i \not \in G_{i-1}$ (by minimality of the generating set) and that $j_i < i$ (by definition). We also explain why $\lambda_i$ is well-defined: Since $\f^{\kappa_i}(P_i) \in H_{j_i} \subseteq \calObar(P_{j_i})$, one of $\pm \f^{\kappa_i}(P_i)$ lies in the orbit of $P_{j_i}$. We cannot have $\f^{\kappa_i}(P_i) \in \calO(P_{j_i})$, since this would then imply that one of $\pm \f^{\kappa_i - 1}(P_i)$ was in $\calO(P_{j_i})$, contradicting the fact that $\f^{\kappa_i - 1}(P_i) \not\in \calObar(P_{j_i})$ by minimality of $\kappa_i$. Therefore $-\f^{\kappa_i}(P_i) \in \calO(P_{j_i})$, so $\lambda_i$ is well-defined.

We are now ready to give a necessary and sufficient condition for two admissible graphs to be isomorphic. In the proof, we will use the fact that an isomorphism of directed graphs $\phi: G \to G'$, with $G$ and $G'$ admissible, necessarily commutes with $\f$. Here we are abusing notation and letting $\f$ denote the map on $G$ as well as on $G'$.

\begin{lem}\label{lem:graph_iso}
Let $G$ be an admissible graph minimally generated by $\P := \{P_1,\ldots,P_n\}$, and for each $i \in \{1,\ldots,n\}$ define $G_i$, $H_i$, $\kappa_i$, $j_i$, and $\lambda_i$ as in \eqref{eq:subgraphs} and \eqref{eq:indices}. Let $G'$ be an admissible graph generated by $\P' := \{P_1',\ldots,P_n'\}$, and define $G_i'$ and $H_i'$ as in \eqref{eq:subgraphs} for each $i \in \{1,\ldots,n\}$. Then there is an isomorphism of directed graphs $G \overset{\sim}{\longrightarrow} G'$ mapping $P_i \mapsto P_i'$ if and only if the following conditions are satisfied for all $i \in \{1,\ldots,n\}$:
	\begin{enumerate}
		\item If $i \in \calD(\P)$, then $\calO(P_i')$ is disjoint from $G_{i-1}'$, and $P_i$ and $P_i'$ have the same preperiodic portrait $(M_i,N_i)$.
		\item If $i \in \calDbar(\P)$, then $\calO(P_i')$ intersects $G_{i-1}'$, and
				\[
					\f^{\kappa_i}(P_i') = -\f^{\lambda_i}(P_{j_i}').
				\]
	\end{enumerate}
\end{lem}

\begin{rem}
The hypotheses of Lemma~\ref{lem:graph_iso} do not, \emph{a priori}, require $\P'$ to be a \emph{minimal} generating set for $G'$. However, it is an immediate consequence of the lemma that if conditions (A) and (B) hold, then $\P'$ is indeed minimal.
\end{rem}

\begin{proof}[Proof of Lemma~\ref{lem:graph_iso}]

The forward implication is clear, so we suppose that conditions (A) and (B) are satisfied and show that there is a graph isomorphism $G \to G'$ mapping $P_i \mapsto P_i'$ for each $i \in \{1,\ldots,n\}$. We proceed by induction on $n$.

The $n = 1$ case is immediate, since there is a unique admissible graph up to isomorphism generated by a single vertex of a given preperiodic portrait. Now suppose that $n > 1$. By the induction hypothesis, there is a graph isomorphism $\phi : G_{n-1} \overset{\sim}{\longrightarrow} G_{n-1}'$ that maps $P_i \mapsto P_i'$ for each $i \in \{1,\ldots,n-1\}$.

First, suppose that $n \in \calD(\P)$, so that $\calO(P_n)$ is disjoint from $G_{n-1}$. Then the elements of $H_n$ are the vertices of an admissible subgraph (which we also call $H_n$) of $G$, so $G$ is the disjoint union $G_{n-1} \sqcup H_n$ of two admissible graphs. By condition (A), the same must be true of $G'$, so we write $G' = G_{n-1}' \sqcup H_n'$ as a disjoint union of admissible graphs. By the $n = 1$ case, there is an isomorphism $H_n \overset{\sim}{\longrightarrow} H_n'$ mapping $P_n \mapsto P_n'$, so we may extend $\phi$ to obtain the desired isomorphism $G \overset{\sim}{\longrightarrow} G'$.

Now suppose that $n \in \calDbar(\P)$, so that $\calO(P_n)$ intersects $G_{n-1}$. Then
	\[ H_n = \{\pm \f^k(P_n) \ | \ 0 \le k \le \kappa_n - 1\}, \]
and, as sets of vertices, $G$ is the disjoint union $G = G_{n-1} \sqcup H_n$. To show that $\phi$ extends to an isomorphism $G \overset{\sim}{\longrightarrow} G'$ that maps $P_n \mapsto P_n'$, it suffices to show that $\f^{\kappa_n}(P_n')$ is the \emph{first} vertex in the orbit of $P_n'$ to lie in $G_{n-1}'$. Indeed, this would imply that
	\[ H_n' = \{\pm \f^k(P_n') \ | \ 0 \le k \le \kappa_n - 1\}, \]
and $G' = G_{n-1}' \sqcup H_n'$. The isomorphism $\phi : G_{n-1} \to G_{n-1}'$ could therefore be easily extended to an isomorphism $G \to G'$ since we have $\f^{\kappa_n}(P_n') = -\f^{\lambda_n}(P_{j_n}')$ by condition (B). 

Suppose to the contrary that $\f^{\kappa_n - 1}(P_n') \in G_{n-1}'$. Then $\f^{\kappa_n-1}(P_n') = \pm \f^k(P_j')$ for some $j \in \{1,\ldots,n-1\}$ and $k \in \bbN$, which implies that $\f^{\kappa_n}(P_n') = \f^{k+1}(P_j')$. By condition (B), we have
	\[
		\f^{\lambda_n}(P_{j_n}') = -\f^{\kappa_n}(P_n') = -\f^{k+1}(P_j').
	\]
Since $\phi^{-1}$ takes $P_j' \mapsto P_j$ and $P_{j_n}' \mapsto P_{j_n}$, we see that we must have
	\[
		\f^{\lambda_n}(P_{j_n}) = -\f^{k+1}(P_j),
	\]
and therefore $\f^{\kappa_n}(P_n) = -\f^{\lambda_n}(P_{j_n}) = \f^{k+1}(P_j)$. This implies that
	\[
		\f^{\kappa_n - 1}(P_n) \in \{ \pm \f^k(P_j)\} \subseteq \calObar(P_j) \subseteq G_{n-1},
	\]
contradicting the fact that $\f^{\kappa_n}(P_n)$ is the first element of $\calO(P_n)$ to lie in $G_{n-1}$.
\end{proof}

Suppose now that $G$ is nearly admissible in the sense of Definition~\ref{defn:ramified}. We define a generating set for $G$ in exactly the same way that we did for admissible graphs, with the understanding that if $P$ is the unique vertex for which $\f(P)$ has a single preimage, then $-P = P$. If $G$ is admissible and $G'$ is nearly admissible, then we say that $G'$ is \textbf{nearly isomorphic to $G$} if there is a minimal generating set $\{P_1,\ldots,P_n\}$ for $G$ and a generating set $\{P_1',\ldots,P_n'\}$ for $G'$ such that conditions (A) and (B) from Lemma~\ref{lem:graph_iso} are satisfied.

\subsection{Generalized dynamical modular curves}\label{sub:gen_dmc}

We are now ready to define, for each admissible graph $G$, a curve $Y_1(G)$ whose points generically correspond to quadratic polynomials $f_c$ together with a collection of points $\{\alpha_1,\ldots,\alpha_n\}$ that generate a subgraph of $G(f_c,\bbC)$ isomorphic to $G$. Now fix an admissible graph $G$, let $\P := \{P_1,\ldots,P_n\}$ be a minimal generating set for $G$, and for each $i \in \{1,\ldots,n\}$ let $(M_i,N_i)$ be the portrait of $P_i$. For each $i \in \calDbar := \calDbar(\P)$, let $j_i$, $\kappa_i$, and $\lambda_i$ be defined as in \eqref{eq:indices}. To each generator $P_i$ we associate a polynomial $\Psi_i(\x,t) \in \bbZ[\x,t]$ as follows, where we write $\x$ to represent the tuple $(x_1,\ldots,x_n)$ of variables:

	\begin{equation}\label{eq:Psi}
		\Psi_i(\x,t) :=
			\begin{cases}
				\Phi_{M_i,N_i}(x_i,t) , &\mbox{ if } i \in \calD; \\
				f_t^{\kappa_i}(x_i) + f_t^{\lambda_i}(x_{j_i}) , &\mbox{ if } i \in \calDbar.
			\end{cases}
	\end{equation}
We define $Y_G$ to be the reduced subscheme of $\bbA^{n+1}$ associated to the vanishing locus of the polynomials $\Psi_1,\ldots,\Psi_n$. Since $Y_G$ is defined by $n$ polynomial equations, each component of $Y_G$ has dimension at least one. On the other hand, the morphism $Y_G \to \bbA^1$ defined by $(\x,t) \mapsto t$ is easily seen to have finite fibers, so each component of $Y_G$ has dimension one. Thus $Y_G$ is a (possibly reducible) curve over $\bbQ$.

\begin{rem}\label{rem:3cycles}
The polynomials $\Psi_i$ have coefficients in $\bbZ$, so $Y_G$ may be defined over any field. Moreover, since $\Psi_i$ is monic in $x_i$ and $t$ (as well as $x_{j_i}$ in the case that $i \in \calDbar$), the argument that $Y_G$ has dimension one is valid over any field. We may therefore consider $Y_G$ as a curve over $\bbZ$, in the sense that $Y_G$ is a $\bbZ$-scheme with one-dimensional fibers.
\end{rem}

Unfortunately, the curve $Y_G$ is generally too large to serve as the dynamical modular curve for the admissible graph $G$. We illustrate this with the following example:

\begin{ex}\label{ex:3cycles}
Let $G$ be the graph minimally generated by two points of period 3. Note that the corresponding 3-cycles must be disjoint, as shown in Figure~\ref{fig:3and3}, since otherwise the graph could be generated by a single point of period three. The curve $Y_G$ is then defined by the equations
	\[ \Phi_3(x_1,t) = \Phi_3(x_2,t) = 0, \]
so that $Y_G$ is the fiber product of $Y_1(3)$ with itself relative to the map to $\bbA^1$ obtained by projecting onto the $t$-coordinate. If $\alpha$ is a point of period 3 for $f_c$, then $\Phi_3(\alpha,c) = 0$, and therefore $(\alpha,\alpha,c)$ is a point on $Y_G$. However, the point $(\alpha,\alpha,c)$ only carries the information of a \emph{single} 3-cycle, not the two 3-cycles that constitute $G$. Geometrically, the issue is that there is an irreducible component of $Y_G$ defined by the additional equation $x_1 = x_2$, as well as components given by $x_1 = f_t(x_2)$ and $x_1 = f_t^2(x_2)$, whose points $(\alpha_1,\alpha_2,c)$ fail to determine two distinct 3-cycles under $f_c$. Though it is not apparent here, in this case $Y_G$ has one remaining irreducible component, and it is this component that will serve as $Y_1(G)$.
\end{ex}

\begin{figure}
\centering
%    \begin{tikzpicture}[scale=.65]
%\tikzset{vertex/.style = {}}
%\tikzset{every loop/.style={min distance=10mm,in=45,out=-45,->}}
%\tikzset{edge/.style={decoration={markings,mark=at position 1 with %
%    {\arrow[scale=1.5,>=stealth]{>}}},postaction={decorate}}}
%% vertices
%\node[vertex] (3a) at  (0,0) {$\bullet$};
%\node[vertex] (3b) at  (0,2) {$\bullet$};
%\node[vertex] (3c) at  (1.7,1) {$\bullet$};
%\node[vertex] (13a) at  (-1,-1.7) {$\bullet$};
%\node[vertex] (13b) at (-1,3.7) {$\bullet$};
%\node[vertex] (13c) at (3.7, 1) {$\bullet$};
%%
%\node[vertex] (3a') at  (5,0) {$\bullet$};
%\node[vertex] (3b') at  (5,2) {$\bullet$};
%\node[vertex] (3c') at  (6.7,1) {$\bullet$};
%\node[vertex] (13a') at  (4,-1.7) {$\bullet$};
%\node[vertex] (13b') at (4,3.7) {$\bullet$};
%\node[vertex] (13c') at (8.7, 1) {$\bullet$};
%% edges
%\draw[edge] (3a) to (3b);
%\draw[edge] (3b) to (3c);
%\draw[edge] (3c) to (3a);
%\draw[edge] (13a) to (3a);
%\draw[edge] (13b) to (3b);
%\draw[edge] (13c) to (3c);
%%
%\draw[edge] (3a') to (3b');
%\draw[edge] (3b') to (3c');
%\draw[edge] (3c') to (3a');
%\draw[edge] (13a') to (3a');
%\draw[edge] (13b') to (3b');
%\draw[edge] (13c') to (3c');
%\end{tikzpicture}
	\includegraphics{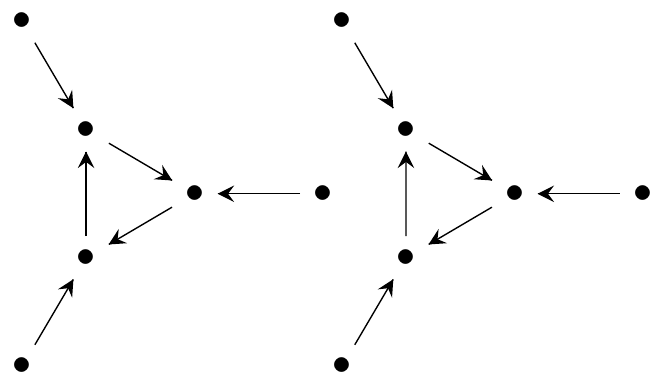}
\caption{The admissible graph minimally generated by two points of period 3}
\label{fig:3and3}
\end{figure}

From the above example, we see that if $P_i$ and $P_j$ are generators lying in disjoint components in $G$, then $\alpha_i$ and $\alpha_j$ should be forced to have disjoint orbits under $f_c$. We therefore impose additional conditions in order to define $Y_1(G)$.

\begin{defn}\label{defn:Y1(G)}
Let $G$ be an admissible graph, minimally generated by $\P = \{P_1,\ldots,P_n\}$, and let $\calD = \calD(\P)$ be as defined above. We define $Y_1(G)$ to be the closure (in $\bbA^{n+1}$) of the Zariski open subset of $Y_G$ determined by the conditions
	\begin{equation} \label{eq:disjoint_orbits}
		\begin{split}
			f_t^{M_i}(x_i) \ne f_t^{M_j + k}(x_j) &\mbox{ for all $0 \le k \le N_i - 1$,}\\
			&\mbox{ for all $i,j \in \calD$ with $i > j$ and $N_i = N_j$}.
		\end{split}
	\end{equation}
\end{defn}

In other words, $Y_1(G)$ is the union of the irreducible components of $Y_G$ on which, \emph{generically}, the conditions of \eqref{eq:disjoint_orbits} are satisfied. In fact, what remains is a single irreducible component  --- this is the content of Theorem~\ref{thm:main_irred}. We then define $U_1(G)$ to be the Zariski open subset of $Y_1(G)$ determined by \eqref{eq:disjoint_orbits} (that is, we remove the points of intersection of $Y_1(G)$ with the excised components of $Y_G$) as well as the following conditions from \eqref{eq:phiMNconditions}:	
	\begin{equation}\label{eq:right_portrait}
		\begin{split}
		\Phi_{M_i,n}(x_i,t) \ne 0 &\mbox{ for all $1 \le n < N_i$ with $n \mid N_i$, for all $i \in \calD$};\\
		\Phi_{m,N_i}(x_i,t) \ne 0 &\mbox{ for all $0 \le m < M_i$, for all $i \in \calD$}.
		\end{split}
	\end{equation}
As in previous sections, $Y_1(G)$ is the affine closure of $U_1(G)$ in $\bbA^{n+1}$. We define $X_1(G)$ to be the normalization of the projective closure of $Y_1(G)$. Finally, we denote by $\pi_G$ the morphism $X_1(G) \to \bbP^1$ that maps $(\x,t) \mapsto t$; this should be viewed as an analogue of the map from the classical modular curve $\Xell_1(N)$ to $\bbP^1$ given by the $j$-invariant.

\begin{rem}
It is apparent from the definition that $Y_1(G)$ depends on the order of the vertices in the minimal generating set $\P$. However, it is not difficult to see that a permutation of the vertices in $\P$ yields an isomorphic curve with a different embedding into $\bbA^{n+1}$. For example, this is an immediate consequence of Proposition~\ref{prop:right_graph} below.
\end{rem}

\begin{rem}
The empty graph $\emptyset$ (that is, the graph with no vertices) satisfies the conditions of admissibility, and is minimally generated by the empty set. In this case, we have 
\[ U_1(\emptyset) = Y_1(\emptyset) = \Spec \bbZ[t] \cong \bbA^1_\bbZ,\]
which is simply the parameter space of maps $f_c$ without additional preperiodic structure.
\end{rem}

We will refer to the curves $X_1(G)$, $Y_1(G)$, and $U_1(G)$ generally as \textbf{dynamical modular curves}. We now justify this terminology:

\begin{prop}\label{prop:right_graph}
Let $K$ be a field of characteristic zero. Let $G$ be an admissible graph with minimal generating set $\P = \{P_1,\ldots,P_n\}$. Then $(\alpha_1,\ldots,\alpha_n,c) \in U_1(G)(K)$ if and only if $\{\alpha_1,\ldots,\alpha_n\}$ generates a subgraph of $G(f_c,K)$ that is isomorphic or nearly isomorphic to $G$ via an identification $P_i \mapsto \alpha_i$.
\end{prop}

\begin{proof}
Suppose first that $(\balpha,c) := (\alpha_1,\ldots,\alpha_n,c) \in U_1(G)(K)$. The relations defining $Y_G$ imply that each of $\alpha_1,\ldots,\alpha_n$ is preperiodic under $f_c$, so we let $G'$ be the subgraph of $G(f_c,K)$ generated by $\{\alpha_1,\ldots,\alpha_n\}$. Since $G'$ is either admissible or nearly admissible, it will suffice to show that the graphs $G$ and $G'$ satisfy conditions (A) and (B) of Lemma~\ref{lem:graph_iso}.

Suppose $i \in \calD$, and let $(M_i,N_i)$ be the preperiodic portrait of the vertex $P_i$. Then $\Psi_i(\balpha,c) = \Phi_{M_i,N_i}(\alpha_i,c) = 0$; combining this with \eqref{eq:right_portrait} tells us that $\alpha_i$ also has portrait $(M_i,N_i)$. Furthermore, condition \eqref{eq:disjoint_orbits} guarantees that the orbit of $\alpha_i$ under $f_c$ is disjoint from the orbits of $\alpha_1,\ldots,\alpha_{i-1}$ under $f_c$. Therefore condition (A) of Lemma~\ref{lem:graph_iso} is satisfied.

Now suppose that $i \in \calDbar$. Then
	\begin{equation}\label{eq:Dbar}
		0 = \Psi_i(\balpha,c) = 
				f_c^{\kappa_i}(\alpha_i) + f_c^{\lambda_i}(\alpha_{j_i}).			
	\end{equation}
Since $f_c^{\kappa_i}(\alpha_i) = -f_c^{\lambda_i}(\alpha_{j_i})$, condition (B) of Lemma~\ref{lem:graph_iso} is satisfied. We conclude that $G'$ is isomorphic or nearly isomorphic to $G$ via an identification $P_i \mapsto \alpha_i$, as claimed.

To prove the converse, suppose that $\{\alpha_1,\ldots,\alpha_n\} \subseteq \PrePer(f_c,K)$ generates a subgraph $G' \subseteq G(f_c,K)$ which is isomorphic or nearly isomorphic to $G$ via an identification $P_i \mapsto \alpha_i$. In other words, we assume that $G$ and $G'$ satisfy conditions (A) and (B) from Lemma~\ref{lem:graph_iso}.

If $i \in \calD$, then condition (A) implies that the orbit of $\alpha_i$ is disjoint from $G_{i-1}'$ (hence \eqref{eq:disjoint_orbits} is satisfied), and both $P_i$ and $\alpha_i$ have the same preperiodic type $(M_i,N_i)$. Therefore $\Phi_{M_i,N_i}(\alpha_i,c) = 0$, $\Phi_{M_i,n}(\alpha_i,c) \ne 0$ for all $0 \le n < N_i$ with $n \mid N_i$, and $\Phi_{m,N_i}(\alpha_i,c) \ne 0$ for all $0 \le m < M_i$.

On the other hand, if $i \in \calDbar$, then condition (B) asserts that the orbit of $\alpha_i$ intersects $G_{i-1}'$; moreover, we have $
		\f^{\kappa_i}(P_i) = 
				-\f^{\lambda_i}(P_{j_i}),
	$
so the identification $P_i \mapsto \alpha_i$ implies that
	\[
		f_c^{\kappa_i}(\alpha_i) + f_c^{\lambda_i}(\alpha_{j_i}) = 0.
	\]
Combined with the previous paragraph, this shows that $(\balpha,c) \in U_1(G)(K)$.
\end{proof}

\begin{cor}\label{cor:finite_ramified}
Let $K$ be a field of characteristic zero, and let $G$ be an admissible graph. For all but finitely many points $(\alpha_1,\ldots,\alpha_n,c) \in X_1(G)(K)$, $\{\alpha_1,\ldots,\alpha_n\}$ generates a subgraph of $G(f_c,K)$ isomorphic to $G$.
\end{cor}

\begin{proof}
It suffices to show that the statement is true with $X_1(G)$ replaced by $U_1(G)$, since the latter is Zariski dense in the former. Let $N$ be the largest cycle length appearing in $G$, and let $M$ be the largest preperiod of a vertex in $G$. Then every vertex of $G$ has portrait $(m,n)$ with $m \le M$ and $n \le N$, and the same is true for any graph nearly isomorphic to $G$.

Now suppose $(\alpha_1,\ldots,\alpha_n,c) \in U_1(G)(K)$ is such that $\{\alpha_1,\ldots,\alpha_n\}$ does not generate a subgraph of $G(f_c,K)$ isomorphic to $G$. By Proposition~\ref{prop:right_graph}, $\{\alpha_1,\ldots,\alpha_n\}$ must therefore generate a subgraph of $G(f_c,K)$ which is nearly isomorphic to $G$. This implies that $0$ has portrait $(m,n)$ for $f_c$ for some $m \le M$ and $n \le N$, hence $\Phi_{m,n}(0,c) = 0$ for some $m \le M$ and $n \le N$. The set of such $c \in K$ is clearly finite, and since the projection map $\pi_G: X_1(G) \to \bbP^1$ is finite, the result follows.
\end{proof}

\subsection{The curves $X_1(N^{(r)})$}\label{sec:equiv}
Let $N \in \bbN$, and let $1 \le r \le R := R(N)$. Let $G$ be the admissible graph minimally generated by $r$ points of period $N$. Note that the minimality condition ensures that the $r$ points lie in distinct components of $G$. We define
	\[
		X_1(N^{(r)}) = X_1(\underbrace{N,\ldots,N}_{\text{$r$ times}}) := X_1(G),
	\]
and we similarly define $Y_1(N^{(r)})$ and $U_1(N^{(r)})$.

Using the definitions of the previous section, we see that $U_1(N^{(r)}) \subset \bbA^{r+1}$ is defined by the conditions
	\begin{align*}
		\Phi_N(x_i,t) = 0 &\text{ for all } 1 \le i \le r;\\
		x_i \ne f_t^k(x_j) &\text{ for all } 1 \le j < i \le r,\\
			&\text{ for all } 0 \le k \le N - 1;
	\end{align*}
and $Y_1(N^{(r)})$ is the closure of $U_1(N^{(r)})$ in $\bbA^{r+1}$. For the purposes of the next section, we now provide an equivalent definition of $Y_1(N^{(r)})$. First, we recall the following theorem due to Bousch:

\begin{thm}[{Bousch \cite[\textsection 3, Thm. 1]{bousch:1992}}]\label{thm:bousch}
For each $N \in \bbN$, the polynomial $\Phi_N(x,t)$ is irreducible over $\bbC$. If we let $L_N$ be the splitting field of $\Phi_N(x,t)$ over $\bbC(t)$, then the Galois group $\Gal(L_N/\bbC(t))$ consists of all permutations of the roots of $\Phi_N(x,t) \in \bbC(t)[x]$ that commute with $x \mapsto f_t(x)$.
\end{thm}

Morton \cite{morton:1996} remarks that $\Gal(L_N/\bbC(c))$ is isomorphic to the \emph{wreath product} of the cyclic group $\bbZ/N\bbZ$ with the symmetric group $S_R$, which is a certain semi-direct product $(\bbZ/N\bbZ)^R \rtimes S_R$. See \cite[\textsection 3.9]{silverman:2007} for a discussion of the wreath product in this setting. The Galois group is allowed to permute the $R$ different $N$-cycles \emph{as sets} by any permutation in $S_R$, and is also allowed to act by a cyclic permutation within each $N$-cycle independently. In particular, we have $[L_N : \bbC(t)] = R!N^R$.

Let $x_1,\ldots,x_R$ be roots of $\Phi_N(x,t)$ in an algebraic closure of $\bbC(t)$ with pairwise disjoint orbits under $f_t$; the full set of roots of $\Phi_N(x,t)$ is then $\{f_t^k(x_i): 1 \le i \le R, \ 0 \le k \le N - 1\}$. For each $r \in \{1,\ldots,R\}$, we define a polynomial $g_{N,r}$ by
	\[
		g_{N,r}(x) := \prod_{i=r}^R \prod_{k=0}^{N-1} \left(x - f_t^k(x_i)\right).
	\]
Note that, since $\displaystyle \Phi_N(x,t) = g_{N,r}(x) \cdot \prod_{j=1}^{r-1} \prod_{k=0}^{N-1} \left(x - f_t^k(x_j)\right)$, the polynomial $g_{N,r}(x)$ has coefficients in $\bbC[x_1,\ldots,x_{r-1},t]$; we will sometimes write $g_{N,r}(x_1,\ldots,x_{r-1},x,t)$ to make this clear. Even better, we have the following:

\begin{lem}
Let $F_0 := \bbC(t)$, and for each $r \in \{1,\ldots,R\}$, define the field extension $F_r := \bbC(x_1,\ldots,x_r,c)$. Then $g_{N,r}$ is irreducible over $F_{r-1}$ for each $r \in \{1,\ldots,R\}$.
\end{lem}

\begin{proof}
First, we note that $F_R$ is the splitting field of $\Phi_N(x,t)$, so $F_R$ is equal to the field $L_N$ described above. Now, any two periodic points in the same orbit under $f_t$ generate the same extension of $\bbC(t)$, since each can be written as an iterate of the other under $f_t$. Thus we have
	\begin{equation}\label{eq:degree}
		[F_R : \bbC(t)] \ge \prod_{r=1}^R \deg g_{N,r}(x),
	\end{equation}
with equality if and only if each $g_{N,r}$ is irreducible over $F_{r-1}$. Since $\deg g_{N,r}(x) = (R - r + 1)N$, the product on the right hand side of \eqref{eq:degree} is equal to
	\[
		\prod_{r=1}^R \left[(R - r + 1)N\right] = R!N^R = [F_R : \bbC(t)],
	\]
completing the proof.
\end{proof}

\begin{cor}
Let $C_{N,r} \subset \bbA^{r+1}$ be the curve defined by
\[
	g_{N,1}(x_1,t) = g_{N,2}(x_1,x_2,t) \cdots = g_{N,r}(x_1,\ldots,x_r,t) = 0.
\]
Then $C_{N,r}$ is irreducible over $\bbC$ with function field $F_r$.
\end{cor}

We conclude this section by verifying that $C_{N,r}$ is equal to the dynamical modular curve $Y_1(N^{(r)})$.

\begin{prop}\label{prop:Y1Nn_irred}
Let $N \ge 1$ and $1 \le r \le R(N)$ be integers. Then $Y_1(N^{(r)}) = C_{N,r}$. In particular, $Y_1(N^{(r)})$ is irreducible over $\bbC$ with function field $F_r$.
\end{prop}

\begin{proof}
Since $Y_1(N^{(r)})$ is the closure of $U_1(N^{(r)})$ in $\bbA^{r+1}$, and since $C_{N,r}$ is an irreducible closed subscheme of $\bbA^{r+1}$, it suffices to show that $U_1(N^{(r)}) \subseteq C_{N,r}$.

Let $(\balpha,c) = (\alpha_1,\ldots,\alpha_r,c) \in U_1(N^{(r)})(\bbC)$. First, this implies that $\Phi_N(\alpha_i,c) = 0$ for each $i \in \{1,\ldots,r\}$, which allows us to write
	\begin{equation}\label{eq:gNm}
		0 = \Phi_N(\alpha_i,c) = g_{N,i}(\alpha_1,\ldots,\alpha_{i-1},\alpha_i,c) \cdot \prod_{j=1}^{i-1}\prod_{k=0}^{N-1} \left(\alpha_i - f_c^k(\alpha_j)\right).
	\end{equation}
Moreover, the fact that $(\balpha,c) \in U_1(N^{(r)})$ implies that $\alpha_i \ne f_c^k(\alpha_j)$ for each $1 \le j \le i - 1$ and $0 \le k \le N - 1$, hence $g_{N,i}(\alpha_1,\ldots,\alpha_i,c) = 0$ by \eqref{eq:gNm}. Since this holds for all $i \in \{1,\ldots,r\}$, we conclude that $(\balpha,c) \in C_{N,r}(\bbC)$.
\end{proof}

\section{Properties of dynamical modular curves}\label{sec:dyn_mod}

\subsection{Morphisms between dynamical modular curves}\label{sec:directed_system}

For a given admissible graph $G$, there is a natural morphism $\pi_G : X_1(G) \to \bbP^1$ given by projection onto the $t$-coordinate; this map is analogous to the map from the classical modular curve $\Xell_1(N)$ to the $j$-line. We now show that if $G$ and $H$ are admissible graphs with $G \supseteq H$, then there is a natural morphism $X_1(G) \to X_1(H)$. We first prove this assertion in two special cases.

\begin{lem}\label{lem:graph_morphisms}
Let $H$ be an admissible graph.
	\begin{enumerate}
		\item (Adding a cycle) Suppose $H$ contains exactly $r$ connected components terminating in cycles of length $N$, with $0 \le r < R(N)$. Let $G$ be the admissible graph obtained from $H$ by adding a new point of period $N$ (and therefore a full $N$-cycle together with its portrait $(1,N)$ preimages). Then there exists a dominant morphism $\pi : X_1(G) \to X_1(H)$ of degree $N(R(N) - r)$. In particular, $\deg \pi \ge 2$ unless $N = r = 1$.
		\item (Adding preimages to an initial point) Let $P_0$ be an initial point for $H$. Let $G$ be the admissible graph obtained by appending two preimages $P$ and $-P$ to $P_0$, as in Figure~\ref{fig:adding_preimages}. Then there is a dominant morphism $\pi : X(G) \to X(H)$ of degree two.
	\end{enumerate}
Moreover, in both cases we have $\pi_G = \pi_H \circ \pi$.
\end{lem}

\begin{figure}
\centering
%    \begin{tikzpicture}[scale=.65]
%\tikzset{vertex/.style = {}}
%\tikzset{every loop/.style={min distance=10mm,in=45,out=-45,->}}
%\tikzset{edge/.style={decoration={markings,mark=at position 1 with %
%    {\arrow[scale=1.5,>=stealth]{>}}},postaction={decorate}}}
%% vertices
%\node[vertex, label={[label distance=.1mm]$P_0$}] (P0) at  (-3.5,0) {$\bullet$};
%\node[vertex] (dots) at  (-1,0) {\Large $\cdots$};
%\node[vertex, label=$P$] (P) at  (4.3,1) {$\bullet$};
%\node[vertex, label=$-P$] (-P) at  (4.3,-1) {$\bullet$};
%\node[vertex, label=$P_0$] (P0') at  (6,0) {$\bullet$};
%\node[vertex] (dots') at  (8.5,0) {\Large $\cdots$};
%\node[vertex] (blank) at (0,0) {};
%\node[vertex] (blank') at (3,0) {};
%\node[vertex] (H) at (-2,-2) {\huge $H$};
%\node[vertex] (G) at (6,-2) {\huge $G$};
%% edges
%\draw[edge] (P0) to (dots);
%\draw[edge] (blank) to[bend left=20] (blank');
%\draw[edge] (P) to (P0');
%\draw[edge] (-P) to (P0');
%\draw[edge] (P0') to (dots');
%\end{tikzpicture}
	\includegraphics{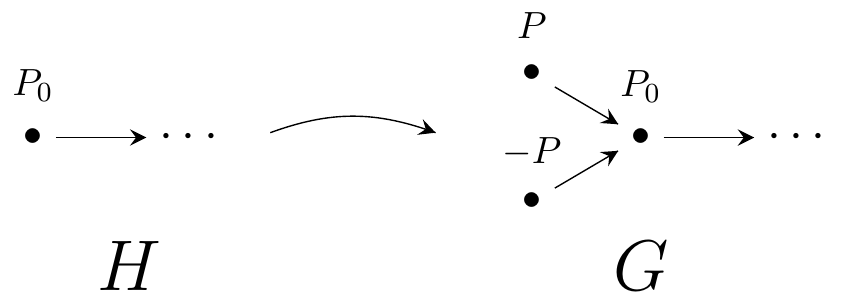}
\caption{Appending two preimages to the vertex $P_0$ in $H$}
\label{fig:adding_preimages}
\end{figure}

\begin{rem}\label{rem:irred}
Since we have not yet shown that dynamical modular curves are irreducible in general, we take the degree of a morphism simply to be the number of preimages of a typical point in the codomain.
\end{rem}

\begin{proof}[Proof of Lemma~\ref{lem:graph_morphisms}]
First, let $G$ and $H$ be as in part (A), and let $R := R(N)$ and $D := D(N) = RN$. Say $H$ is minimally generated by $\{P_1,\ldots,P_n\}$, in which case $G$ is minimally generated by $\{P,P_1,\ldots,P_n\}$, where $P$ is the new point of period $N$. Consider the natural map
	\begin{align*}
		\pi : X_1(G) &\to X_1(H)\\
		(x,x_1,\ldots,x_n,c) &\mapsto (x_1,\ldots,x_n,c).
	\end{align*}
Now let $(\alpha_1,\ldots,\alpha_n,c) \in U_1(H)(\bbC)$ be such that $f_c$ has precisely $D$ points of period $N$. The set of such points is cofinite in $X_1(H)(\bbC)$, since $X_1(H) \setminus U_1(H)$ contains finitely many points and there are only finitely many $c \in \bbC$ for which $f_c$ has fewer than $D$ points of period $N$. (These are the roots of $\disc_x \Phi_N(x,t) \in \bbC[t]$; see \cite{morton/vivaldi:1995} for a detailed study of discriminants arising in this way.)
	
The preimages of $(\alpha_1,\ldots,\alpha_n,c)$ under $\phi$ are those tuples $(\alpha,\alpha_1,\ldots,\alpha_n,c)$ for which $\alpha$ has period $N$ for $f_c$ and does not lie in the orbit of $\alpha_i$ for any $i \in \{1,\ldots,n\}$. Since $H$ contains $r$ components with $N$-cycles, and therefore contains $rN$ points of period $N$, $\alpha$ must avoid a set of size $rN$. Since $f_c$ has $D = RN$ points of period $N$, that means that $(\alpha_1,\ldots,\alpha_n,c)$ has $(D - rN) = N(R - r)$ preimages. It follows that $\phi$ is dominant of degree $N(R - r)$.

Now let $G$ and $H$ be as in part (B), so that $G$ is obtained from $H$ by appending two preimages $\pm P$ to an initial point $P_0$ of $H$. Fix a minimal generating set $\P := \{P_1,\ldots,P_n\}$ for $H$, and assume that $P_0 \not\in \P$, since otherwise we could replace the generator $P_0$ with $-P_0$. Reordering the vertices in $\P$ if necessary, we will assume that $P_0 \in \calObar(P_1)$. This means one of $\pm P_0$ lies in $\calO(P_1)$; however, since $P_0$ is an initial point of $H$, it must be that $-P_0 = \f^k(P_1)$ for some $k \ge 0$, thus
	\begin{equation}\label{eq:fP}
		\f(P) = -\f^k(P_1)
	\end{equation}
for some $k \ge 0$.

Now, observe that $\Q := \{P,P_1,\ldots,P_n\}$ is a (not necessarily minimal) generating set for $G$. By Lemma~\ref{lem:minimal_subset}, there is some subset $\P' \subseteq \Q$ that minimally generates $G$; note that $P$ must necessarily be in $\P'$, since $P$ is not in the subgraph $H$ generated by $\P$. By Lemma~\ref{lem:more_generators} we must have $|\P'| \ge n$, so we consider two cases depending on $|\P'|$.

\textbf{Case 1:} $|\P'| = n$. In this case, one of the generators $P_i$ must lie in $\calObar(P)$. We claim that $i = 1$, in which case $\P' = \{P,P_2,\ldots,P_n\}$. Indeed, suppose $P_i \in \calObar(P)$. Since $P_i \in H$ but $P \not\in H$, we must have
	\[
		P_i \in \calObar(P) \setminus \{\pm P\} = \calObar(P_0) \subseteq \calObar(P_1).
	\]
By minimality of $\P$, we must have $i = 1$. Since $P_1 \in \calObar(P)$, we can write $P_1 = \pm \f^\ell(P)$ for some $\ell \ge 1$; we therefore define a morphism $\pi: X_1(G) \to X_1(H)$ by
	\[
		(x,x_2,\ldots,x_n,t) \mapsto (\pm f_t^\ell(x),x_2,\ldots,x_n,t).
	\]
By \eqref{eq:fP}, the preimages of a given point $(\alpha_1,\alpha_2,\ldots,\alpha_n,c)$ are precisely those points $(\alpha,\alpha_2,\ldots,\alpha_n,c)$ satisfying $f_c(\alpha) = -f_c^k(\alpha_1)$, which is quadratic in $\alpha$. Thus the degree of this morphism is two.

\textbf{Case 2:} $|\P'| = n + 1$. In this case, we necessarily have $\P' = \{P,P_1,\ldots,P_n\}$, so the natural morphism in this case is
	\begin{align*}
	X_1(G) &\to X_1(H)\\
		(x,x_1,\ldots,x_n,t) &\mapsto (x_1,\ldots,x_n,t).
	\end{align*}
The preimages of $(\alpha_1,\ldots,\alpha_n,c)$ are those points $(\alpha,\alpha_1,\ldots,\alpha_n,c)$ satisfying $f_c(\alpha) = -f_c^k(\alpha_1)$ as in the previous case, so we again have a morphism of degree two. Finally, we note that the $t$-coordinate remains unchanged under each of the morphisms described above, so it follows immediately that $\pi_G = \pi_H \circ \pi$ in each case.
\end{proof}

Now let $G$ be an admissible graph, and let $H$ be an arbitrary admissible subgraph. Then there exists a sequence of admissible subgraphs
	\[ H = H_0 \subset H_1 \subset \cdots \subset H_{n-1} \subset H_n = G, \]
where each $H_k$ is obtained from $H_{k-1}$ by either adding a cycle or adding preimages to an initial point as in Lemma~\ref{lem:graph_morphisms}. It then follows from Lemma~\ref{lem:graph_morphisms} that we get a tower of curves
	\[
		X_1(G) = X_1(H_n) \to X_1(H_{n-1}) \to \cdots \to X_1(H_1) \to X_1(H_0) = X_1(H),
	\]
where each map has degree at least two, except in the case that $H_k$ is obtained from $H_{k-1}$ by adding a second fixed point, in which case the map $H_k \to H_{k-1}$ is an isomorphism. Taking the composition of these maps, we obtain the following:

\begin{prop}\label{prop:graph_morphisms}
Let $G$ be an admissible graph, and let $H$ be an admissible subgraph of $G$. Then there is a dominant morphism $\pi : X_1(G) \to X_1(H)$ satisfying $\pi_G = \pi_H \circ \pi$. Moreover, this morphism has degree at least two if and only if $G \ne H$ and $G$ is not obtained from $H$ by adding a second fixed point to $H$. In particular, if $H$ is strongly admissible and $G \ne H$, then $\deg \pi \ge 2$.
\end{prop}

Now let $G$ be an arbitrary admissible graph, and write $G$ as a disjoint union of admissible subgraphs $G_1,\ldots,G_m$. Proposition~\ref{prop:graph_morphisms} tells us that, for each $i \in \{1,\ldots,m\}$, $\pi_G$ factors through a map $\pi_i : X_1(G) \to X_1(G_i)$. There is therefore a map
	\[
		X_1(G) \to X_1(G_1) \times_{\bbP^1} \cdots \times_{\bbP^1} X_1(G_m),
	\]
where the fiber product is taken relative to the projection maps $\pi_{G_i}$. In general, this morphism need not be a birational map, as illustrated in Example~\ref{ex:3cycles}: if each of $G_1$ and $G_2$ is generated by a single point of period 3, and if $G$ is the disjoint union of $G_1$ and $G_2$ (illustrated in Figure~\ref{fig:3and3}), then the fiber product $X_1(G_1) \times_{\bbP^1} X_1(G_2)$ has four irreducible components, precisely one of which is birational to $X_1(G)$. However, if we require that no two subgraphs $G_i$ have cycle lengths in common, then $X_1(G)$ is, in fact, birational to the fiber product.

\begin{prop}\label{prop:fiber_product}
Let $G$ be an admissible graph, written as a disjoint union $G = \bigsqcup_{i=1}^m G_i$ of admissible subgraphs, where $G_i$ and $G_j$ have no common cycle lengths whenever $i \ne j$. Then $X_1(G)$ is birational to the fiber product $X_1(G_1) \times_{\bbP^1} \cdots \times_{\bbP^1} X_1(G_m)$, where the fiber product is taken relative to the projection maps $\pi_{G_i}$.
\end{prop}

\begin{proof}
Since $U_1(H)$ is dense in $X_1(H)$ for any admissible graph $H$, it suffices to prove the statement for the curves $U_1(\cdot)$ rather than $X_1(\cdot)$.

For each $i = 1,\ldots,m$, let $\P_i := \{P_{i,1},\ldots,P_{i,n_i}\}$ be a minimal generating set for $G_i$. Since the $G_i$ are disjoint, this implies that $\P_1 \cup \cdots \cup \P_m$ is a minimal generating set for $G$. Consider the morphism
	\begin{align*}
	\psi : U_1(G_1) \times_{\bbP^1} \cdots \times_{\bbP^1} U_1(G_m) &\to \bbA^{1 + \sum_{i=1}^m n_i}\\
	((\x_1,t),\ldots,(\x_m,t)) &\mapsto (\x_1,\ldots,\x_m,t),
	\end{align*}
where $\x_i$ is a tuple of indeterminates $(x_{i,1},\ldots,x_{i,n_i})$ corresponding to the generating set $\P_i$ for $G_i$. The map $\psi$ is clearly injective, so it remains only to show that the image of $\psi$ is $U_1(G) \subset \bbA^{1 + \sum n_i}$.

Let $((\balpha_1,c),\ldots,(\balpha_m,c))$ be a point on the fiber product, where $\balpha_i = (\alpha_{i,1},\ldots,\alpha_{i,n_i})$. For each $i = 1,\ldots,m$, the set $\{\alpha_{i,1},\ldots,\alpha_{i,n_i}\}$ generates a subgraph of $G(f_c,\bbC)$ isomorphic or nearly isomorphic to $G_i$ via an identification $P_{i,j} \mapsto \alpha_{i,j}$. At the cost of excluding only finitely many points on the fiber product (see Corollary~\ref{cor:finite_ramified}), we assume that the corresponding subgraph of $G(f_c,\bbC)$ is isomorphic to $G_i$ for each $i = 1,\ldots,m$. Since the graphs $G_1,\ldots,G_m$ have pairwise distinct cycle lengths, the eventual periods of $\alpha_{i,j}$ and $\alpha_{i',j'}$ are distinct for any $i \ne i'$. In particular, this implies that the orbits of $\alpha_{i,j}$ and $\alpha_{i',j'}$ are disjoint whenever $i \ne i'$. Thus the subgraph of $G(f_c,\bbC)$ generated by $\{\balpha_1,\ldots,\balpha_m\}$ is isomorphic to the disjoint union of the graphs $G_i$ --- hence is isomorphic to $G$ --- via an identification $P_{i,j} \mapsto \alpha_{i,j}$ for all $1 \le i \le m$ and $1 \le j \le n_i$. It follows that the point $(\balpha_1,\ldots,\balpha_m,c) = \psi((\balpha_1,c),\ldots,(\balpha_m,c))$ lies on $U_1(G)$ by Proposition~\ref{prop:right_graph}, thus the image of $\psi$ is contained in $U_1(G)$.

It now remains to show that $\psi$ maps onto $U_1(G)$. If $(\balpha_1,\ldots,\balpha_m,c) \in U_1(G)(\bbC)$, then $\{\balpha_1,\ldots,\balpha_m\}$ generates a subgraph of $G(f_c,\bbC)$ isomorphic or nearly isomorphic to $G$ via an identification $P_{i,j} \mapsto \alpha_{i,j}$. By restricting this isomorphism to each of the subgraphs $G_1,\ldots,G_m$, it follows immediately that, for each $i = 1,\ldots,m$, the subgraph of $G(f_c,\bbC)$ generated by $\balpha_i$ is (nearly) isomorphic to $G_i$ via the appropriate identification. Thus $(\balpha_i,c) \in U_1(G_i)$ for each $i = 1,\ldots,m$, and it follows that $(\balpha_1,\ldots,\balpha_m,c) = \psi((\balpha_1,c),\ldots,(\balpha_m,c))$ is in the image of $\psi$.
\end{proof}

\subsection{Irreducibility}\label{subsec:irred}

In this section, we prove our main result:

\begin{mainirredthm}
Let $K$ be a field of characteristic zero. For any admissible graph $G$, the curve $X_1(G)$ is irreducible over $K$.
\end{mainirredthm}

By the Lefschetz principle, it suffices to prove Theorem~\ref{thm:main_irred} only for $K = \bbC$. The proof requires two main steps: the first is to prove the statement for admissible graphs $G$ whose cycles all have the same length, and we then use a simple geometric lemma to extend the result to all admissble graphs. We begin with the following result of Bousch, which extends Theorem~\ref{thm:bousch} to generalized dynatomic polynomials.

\begin{thm}[{Bousch \cite[\textsection 3, Thms. 1 and 4]{bousch:1992}}]\label{thm:bousch_gen}
For each pair of integers $M \ge 0$ and $N \ge 1$, the polynomial $\Phi_{M,N}(x,t)$ is irreducible over $\bbC$. If we let $L_{M,N}$ be the splitting field of $\Phi_{M,N}(x,t)$ over $\bbC(t)$, then $\Gal(L_{M,N}/\bbC(t))$ consists of all permutations of the roots of $\Phi_{M,N}(x,t) \in \bbC(t)[x]$ that commute with $x \mapsto f_t(x)$.
\end{thm}

\begin{rem}\label{rem:auts}
Let $G_{M,N}$ be the admissible graph with ``full level-$(M,N)$ structure," in the sense that $G_{M,N}$ is maximal among all admissible graphs whose cycles all have length $N$ and whose vertices have preperiod at most $M$. Then the Galois group of the splitting field of $\Phi_{M,N}$, described in Theorem~\ref{thm:bousch_gen}, can easily be seen to be isomorphic to $\Aut(G_{M,N})$.
\end{rem}

The following is a consequence of Theorem~\ref{thm:bousch_gen}:

\begin{prop}\label{prop:same_length}
Let $G$ be an admissible graph whose cycles all have the same length. Then $X_1(G)$ is irreducible over $\bbC$.
\end{prop}

\begin{proof}
Let $N$ be the common length of all cycles appearing in $G$. We proceed by induction on the number of points of preperiod greater than one. If $G$ has no vertices of preperiod greater than one, then $G$ may be generated by points of period $N$, hence $X_1(G)$ is irreducible over $\bbC$ by Proposition~\ref{prop:Y1Nn_irred}.

Now let $P$ be a vertex of $G$ of maximal preperiod $M \ge 2$, let $H$ be the subgraph of $G$ obtained by removing $\pm P$, and let $\{P_1,\ldots,P_n\}$ be a minimal generating set for $H$. As in the proof of Lemma~\ref{lem:graph_morphisms}, rearranging the generators if necessary, there is an integer $k \ge 0$ such that $\f(P) = -\f^k(P_1)$. Again referring to the proof of Lemma~\ref{lem:graph_morphisms}, the components of $X_1(G)$ correspond to the extensions of $\bbC(X_1(H)) = \bbC(x_1,\ldots,x_n,t)$ determined by the factors of $h(x) = f_t(x) + f_t^k(x_1)$. Since $h(x)$ is quadratic, it suffices to show that the roots of $h(x)$ do not lie in $\bbC(X_1(H))$.

Since each vertex of $H$ has portrait $(m,N)$ for some $m \le M$, the function field $\bbC(X_1(H))$ and the roots $\pm y$ of $h(z)$ lie in the splitting field $L_{M,N}$ of $\Phi_{M,N}(x,t)$. Since $\pm y$ are portrait-$(M,N)$ points in $L_{M,N}$ with a common image under $f_t$, it follows from Theorem~\ref{thm:bousch_gen} that there is an automorphism $\sigma \in \Gal(L_{M,N}/\bbC(t))$ that transposes $\pm y$ and fixes all other points of portrait $(m,N)$ with $m \le M$. In particular, such $\sigma$ fixes $x_1,\ldots,x_n$ --- hence fixes the field $\bbC(X_1(H))$ --- and transposes $\pm y$. Therefore $\pm y \not \in \bbC(X_1(H))$, and we conclude that $X_1(G)$ is irreducible over $\bbC$.
\end{proof}

For later reference, we record a consequence of Proposition~\ref{prop:same_length} in a special case.

\begin{cor}\label{cor:gal_closure}
Let $G_{M,N}$ be the admissible graph defined in Remark~\ref{rem:auts}. Then $X_1(G_{M,N})$ is irreducible over $\bbC$ with function field $L_{M,N}$. Moreover, the map $X_1(G_{M,N}) \to \bbP^1$ is the Galois closure of the map $X_1((M,N)) \to \bbP^1$, where both maps are projections onto the $t$-coordinate, and its Galois group is $\Gal(L_{M,N}/\bbC(t)) \cong \Aut(G_{M,N})$ .
\end{cor}

\begin{proof}
Irreducibility of $X_1(G_{M,N})$ follows immediately from Proposition~\ref{prop:same_length}. Since $G_{M,N}$ is generated by points of portrait $(M,N)$, we have $\bbC(X_1(G_{M,N})) \subseteq L_{M,N}$; on the other hand, since $G_{M,N}$ has the maximal number of vertices of portrait $(M,N)$, the same is true of $\bbC(X_1(G_{M,N}))$, so in fact $\bbC(X_1(G_{M,N})) = L_{M,N}$. The final assertion follows from the fact that $X_1((M,N))$ is defined by $\Phi_{M,N}(x,t)$ and $L_{M,N}$ is the splitting field of this polynomial over $\bbC(t)$, and the isomorphism $\Gal(L_{M,N}/\bbC(t)) \cong \Aut(G_{M,N})$ is addressed in Remark~\ref{rem:auts}.
\end{proof}

We now consider arbitrary admissible graphs, possibly with cycles of different lengths. We begin with a definition.

\begin{defn}\label{defn:normal}
An admissible graph $G$ is {\bf normal} if it may be written as a disjoint union
	\begin{equation}\label{eq:normal}
		G = \bigsqcup_{i=1}^n G_{M_i,N_i},
	\end{equation}
where the $N_i$ are distinct and $M_i \ge 0$ for each $i \in \{1,\ldots,n\}$. If $G$ is an arbitrary admissible graph, the {\bf normal closure} of $G$ is the minimal normal graph containing $G$. More concretely, if $N_1,\ldots,N_n$ are the cycles appearing in $G$, and if $M_i$ is the maximal preperiod of a vertex whose eventual period is $N_i$, then the normal closure of $G$ is $\bigsqcup_{i=1}^n G_{M_i,N_i}$.
\end{defn}

\begin{rem}
The terminology is motivated by the Galois theory of dynamical modular curves. Once we have proven irreducibility of dynamical modular curves, it is straightforward to verify that if $G'$ is the normal closure of $G$, then $\pi_{G'} : X_1(G') \to \bbP^1$ is the Galois closure of $\pi_G : X_1(G) \to \bbP^1$.
\end{rem}

Now, if $G$ is an arbitrary admissible graph, we have an embedding of $G$ into its normal closure $G'$, and this embedding induces a dominant morphism $X_1(G') \to X_1(G)$. Therefore, in order to prove irreducibility of $X_1(G)$ for arbitrary admissible graphs $G$, it suffices to restrict our attention to {\it normal} graphs.

Let $G$ be as in \eqref{eq:normal}. By Proposition~\ref{prop:fiber_product}, the curve $X_1(G)$ may be written as a fiber product
	\[
		X_1(G) \cong X_1(G_{M_1,N_1}) \times_{\bbP^1} \cdots \times_{\bbP^1} X_1(G_{M_n,N_n}).
	\]
By Proposition~\ref{prop:same_length}, each of the factors $X_1(G_{M_i,N_i})$ is irreducible over $\bbC$; however, the fiber product of irreducible curves need not be irreducible, as shown in Example~\ref{ex:3cycles}. We first state a sufficient condition for certain fiber products of irreducible curves to be irreducible, and we then prove that our curves $X_1(G_{M,N})$ satisfy this condition. The following is an elementary consequence of the Riemann-Hurwitz formula.

\begin{lem}\label{lem:fiber_product_branch}
Let $\calC_1$ and $\calC_2$ be smooth projective curves that are irreducible over $\bbC$, and for each $i \in \{1,2\}$ let $\pi_i : \calC_i \to \bbP^1$ be a finite morphism of degree $d_i \ge 2$. Let $\calB_i \subset \bbP^1(\bbC)$ be the branch locus for $\pi_i$, and suppose that $|\calB_1 \cap \calB_2| \le 1$. Then the fiber product $\calC := \calC_1 \times_{\bbP^1} \calC_2$, taken relative to the maps $\pi_i$, is irreducible over $\bbC$.

Moreover, let $K$, $K_1$, and $K_2$ denote the function fields of $\calC$, $\calC_1$, and $\calC_2$, respectively, and let $L$, $L_1$, and $L_2$ denote their respective Galois closures over $\bbC(t)$. Then
	\[
		\Gal(L/\bbC(t)) \cong \Gal(L_1/\bbC(t)) \times \Gal(L_2/\bbC(t)).
	\]
\end{lem}

\begin{proof}
By introducing a change of coordinates if necessary, we assume that if $\calB_1 \cap \calB_2$ is non-empty, then $\calB_1 \cap \calB_2 = \{\infty\}$. Hence $\calC_1$ and $\calC_2$ have no affine branch points in common.

To prove the lemma, it suffices to show that $L_1$ and $L_2$ --- hence also $K_1$ and $K_2$ --- are linearly disjoint over $\bbC(t)$. Suppose to the contrary that $L_1 \cap L_2 \supsetneq \bbC(t)$. By Riemann-Hurwitz, there are at least two places of $\bbC(t)$ that ramify in $L_1 \cap L_2$, hence at least one finite such place $\frakp$. This place $\frakp$ then ramifies in each of the extensions $L_1$ and $L_2$, thus ramifies in each of $K_1$ and $K_2$, contradicting our assumption that the maps $\pi_1$ and $\pi_2$ have no affine branch points in common. Therefore $L_1$ and $L_2$ must be linearly disjoint over $\bbC(t)$, completing the proof.
\end{proof}

\begin{lem}\label{lem:branch}
Fix $c \in \bbC$. There exists at most one $N \ge 1$ for which there exists $M \ge 0$ such that $\Phi_{M,N}(x,c) \in \bbC[x]$ has a multiple root.
\end{lem}

\begin{proof}
Suppose $\alpha \in \bbC$ is a multiple root of the polynomial $\Phi_{M,N}(x,c)$. If $M = 0$, then this simply means that $\alpha$ is a multiple root of $\Phi_N(x,c)$. On the other hand, if $M \ge 1$, then Lemma~\ref{lem:phiMNalt} allows us to write
	\[
		\Phi_{M,N}(x,c) = \Phi_N(-f_c^{M-1}(x),c).
	\]
Now let $y := -f_c^{M-1}(x)$ and $\beta := -f_c^{M-1}(\alpha)$. Applying the chain rule and using the fact that $\alpha$ is a multiple root of $\Phi_{M,N}(x,c)$, we have
	\[
		0 = \left.\frac{\partial\Phi_{M,N}(x,c)}{\partial x}\right|_{x=\alpha} = \left.\frac{\partial \Phi_N(y,c)}{\partial y}\right|_{y=\beta} \cdot \left(-\left(f_c^{M-1}\right)'(\alpha)\right) = - \left.\frac{\partial \Phi_N(y,c)}{\partial y}\right|_{y=\beta} \cdot 2^{M-1} \prod_{k=0}^{M-2} f_c^k(\alpha).
	\]
It follows that either $\Phi_N(y,c) \in \bbC[y]$ has a multiple root, or $f_c^k(\alpha) = 0$ for some $0 \le k \le M - 2$. In the former case, $c$ is a root of a period-$N$ hyperbolic component of the Mandelbrot set $\calM$; in the latter, $0$ is preperiodic for $f_c$ with eventual period ({\it a priori} dividing) $N$. These two cases are mutually exclusive, since roots of hyperbolic components of $\calM$ lie in $\QQbar \setminus \ZZbar$ (\cite[p. 582]{morton/vivaldi:1995}), while parameters $c$ such that $0$ is preperiodic for $f_c$ must lie in $\ZZbar$ (see the proof of \cite[Prop. 4.22]{silverman:2007}), where $\ZZbar$ denotes the ring of algebraic integers.

First, suppose $c$ is a root of a period-$N$ hyperbolic component of $\calM$. Then $c$ cannot also be the root of a period-$n$ component for any $n \ne N$, since no two hyperbolic components of $\calM$ share a common root (\cite[Prop. XIV.5]{douady/hubbard:1985}). Therefore, $\Phi_{m,n}(x,c)$ cannot have a multiple root for any $n \ne N$ and $m \ge 0$.

Now, suppose $0$ is preperiodic for $f_c$. Since $c$ cannot be a root of a hyperbolic component of $\calM$, it must be that $0$ has eventual period \emph{equal to} $N$ (not just \emph{dividing} $N$). Certainly $0$ cannot also have eventual period equal to $n$ for any $n \ne N$, so $\Phi_{m,n}(x,c)$ cannot have a multiple root for any $n \ne N$ and $m \ge 0$.
\end{proof}

\begin{cor}\label{cor:branch}
Let $N_1 \ne N_2$ be positive integers, and let $M_1,M_2 \ge 0$ be arbitrary. For each $i = 1,2$, let $\calB_i$ be the branch locus of the map $\pi_{G_{M_i,N_i}} : X_1(G_{M_i,N_i}) \to \bbP^1$. Then $\calB_1 \cap \calB_2 = \{\infty\}$.
\end{cor}

\begin{proof}
By Corollary~\ref{cor:gal_closure}, the map $\pi_{G_{M_i,N_i}}$ is the Galois closure of the map
	\begin{align*}
		\pi_i : X_1((M_i,N_i)) &\to \bbP^1\\
			(x,c) &\mapsto c,
	\end{align*}
so it suffices to prove the statement with $\pi_{G_{M_i,N_i}}$ replaced by $\pi_i$.

That each branch locus contains $\infty$ follows from \cite[Prop. 10]{morton:1996}. If $\pi_i$ is ramified over $c \in \bbA^1(\bbC)$, then the polynomial $\Phi_{M_i,N_i}(x,c) \in \bbC[x]$ has a multiple root. Since $N_1 \ne N_2$, the polynomials $\Phi_{M_1,N_1}(x,c)$ and $\Phi_{M_2,N_2}(x,c)$ cannot both have multiple roots by Lemma~\ref{lem:branch}. Therefore, $\calB_1$ and $\calB_2$ share no affine points.
\end{proof}

We may now prove Theorem~\ref{thm:main_irred}.

\begin{proof}[Proof of Theorem~\ref{thm:main_irred}]
Let $G$ be an admissible graph. As mentioned just below Definition~\ref{defn:normal}, it suffices to replace $G$ with its normal closure, so we assume $G$ is normal and write
	\[
	G = \bigsqcup_{i=1}^n G_{M_i,N_i},
	\]
where the $M_i$ are arbitrary nonnegative integers and the $N_i$ are distinct positive integers. By Proposition~\ref{prop:fiber_product}, we can write
	\begin{equation}\label{eq:final_product}
	X_1(G) \cong X_1(G_{M_1,N_1}) \times_{\bbP^1} \cdots \times_{\bbP^1} X_1(G_{M_n,N_n}),
	\end{equation}
where the fiber product is taken relative to the projection maps $\pi_{G_{M_i,N_i}}$. The factors appearing in \eqref{eq:final_product} are irreducible (by Proposition~\ref{prop:same_length}) and have pairwise disjoint affine branch loci (by Corollary~\ref{cor:branch}). It then follows from Lemma~\ref{lem:fiber_product_branch} that $X_1(G)$ is irreducible over $\bbC$, hence over any field of characteristic zero by the Lefschetz principle.
\end{proof}

\section{Realizing admissible graphs over number fields}\label{sec:realize}

In this section, we give an application of Theorem~\ref{thm:main_irred}. Corollary~\ref{cor:admissible} says that if $K$ is a number field, then for all but finitely many $c \in K$, the preperiodic graph $G(f_c,K)$ is admissible. We now prove a converse to this statement.

\begin{thm}\label{thm:realize}
Let $G$ be a strongly admissible graph. Then there exists a number field $K$ and $c \in K$ such that $G(f_c,K) \cong G$.
\end{thm}

Before we prove Theorem~\ref{thm:realize}, we give examples to show that its conclusion may fail if $G$ is {\it nearly} or {\it weakly} admissible, thus demonstrating that the strong admissibility condition is necessary.

\subsection{Nearly and weakly admissible graphs}

\begin{figure}
\centering
%    \begin{tikzpicture}[scale=.65]
%\tikzset{vertex/.style = {}}
%\tikzset{edge/.style={decoration={markings,mark=at position 1 with %
%    {\arrow[scale=1.5,>=stealth]{>}}},postaction={decorate}}}
%\tikzset{loop/.style={min distance=15mm, in=45, out=-45, decoration={markings,mark=at position 0.99 with %
%    {\arrow[scale=1.5,>=stealth]{>}}}, postaction={decorate}}}
%% vertices
%\node[vertex] (0) at  (6, 4) {$\bullet$};
%%
%\node[vertex] (1) at  (6, 2) {$\bullet$};
%%
%\node[vertex] (11) at  (4, 2) {$\bullet$};
%%
%\node[vertex] (21a) at  (2, 3) {$\bullet$};
%\node[vertex] (21b) at  (2, 1) {$\bullet$};
%%
%\node[vertex] (31a) at  (0, 2.5) {$\bullet$};
%\node[vertex] (31b) at  (0, 3.5) {$\bullet$};
%%edges
%\path (0) edge [loop right] node {} (0);
%\path (1) edge [loop right] node {} (1);
%\draw[edge] (11) to (1);
%\draw[edge] (21a) to (11);
%\draw[edge] (21b) to (11);
%\draw[edge] (31a) to (21a);
%\draw[edge] (31b) to (21a);
%\end{tikzpicture}
	\includegraphics{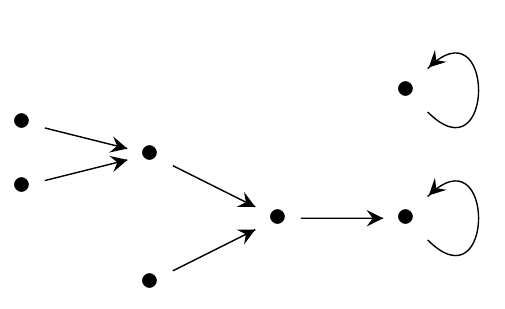}
\caption{A nearly admissible graph $G_1$}
\label{fig:ramified}
\end{figure}

We give one example each of a {\it nearly admissible} and {\it weakly admissible} graph that does not arise as $G(f_c,K)$ for any number field $K$ and $c \in K$.

First, let $G_1$ be the nearly admissible graph appearing in Figure~\ref{fig:ramified}. We claim that there is no number field $K$ with an element $c \in K$ for which $G(f_c,K) \cong G_1$. Suppose that $G(f_c,K)$ has a subgraph isomorphic to $G_1$. Then the critical point $0$ for $f_c$ is fixed, forcing $c = 0$. The second fixed point of $f_0(z) = z^2$ is $1$, and for each $M \ge 1$ the points of portrait $(M,1)$ are precisely the primitive $2^{M}$th roots of unity. Since $G(f_0,K)$ has a vertex of portrait $(3,1)$, $K$ must contain a primitive eighth root of unity, hence $K$ necessarily contains all four primitive eighth roots of unity. It follows that $G(f_0,K)$ has \textit{four} points of portrait $(3,1)$, not just the two appearing in $G_1$, hence $G(f_c,K)$ must {\it properly} contain $G_1$.

\begin{figure}
\centering
	\includegraphics{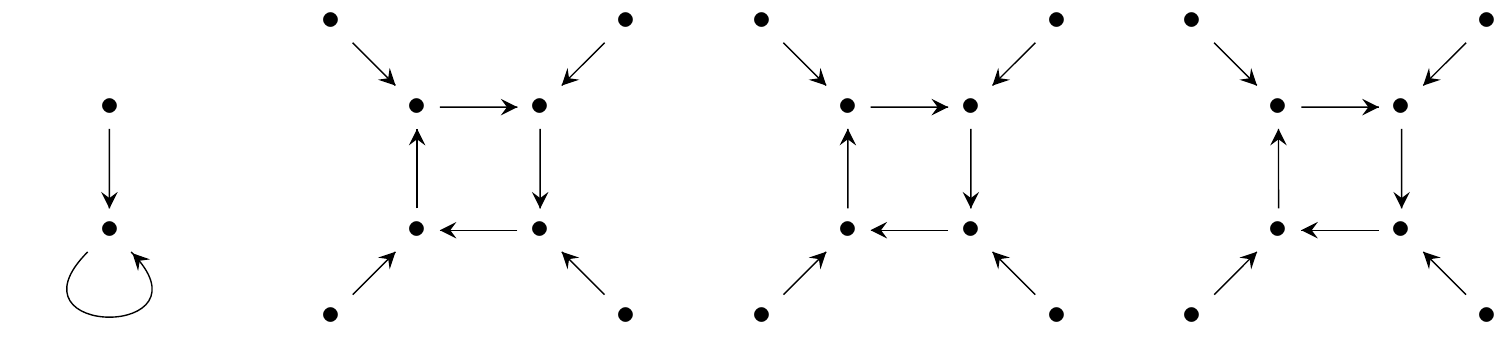}
\caption{A weakly admissible graph $G_2$}
\label{fig:weak}
\end{figure}

Now let $G_2$ be the weakly admissible graph appearing in Figure~\ref{fig:weak} --- it is the graph generated by a single fixed point and three points of period four in distinct orbits. We now claim that $G_2$ is never realized as $G(f_c,K)$ for any number field $K$ and $c \in K$. Again, we suppose that $G(f_c,K)$ has a subgraph isomorphic to $G_2$. If $f_c$ has two fixed points, then certainly $G(f_c,K) \not\cong G_2$, so we assume $f_c$ has only one fixed point; i.e., we take $c = 1/4$. The graph $G_2$ contains $D(4) = 12$ points of period four, so $K$ must contain the splitting field $L$ of $\Phi_4(x,1/4) \in \bbQ[x]$. One can verify using a computer algebra system (we have done so in Magma) that $L$ contains $\sqrt{-3/4}$, which has portrait $(2,1)$ for $f_{1/4}$. Therefore, if $K \supseteq L$, the graph $G(f_c,K)$ must have a point of portrait $(2,1)$, hence properly contains $G_2$.

\subsection{Proof of Theorem~\ref{thm:realize}}

Our proof of Theorem~\ref{thm:realize} uses a result due to Morton and Silverman \cite{morton/silverman:1994} that bounds\footnote{Stronger bounds were subsequently obtained by Benedetto \cite{benedetto:2007}; for our purposes, however, any bound is sufficient.} the periods of points under a rational map $\phi$ in terms of the primes of good reduction for $\phi$. A rational map $\phi(z) \in K(z)$ has \emph{good reduction} at the prime $\frakp$ if the reduction of $\phi$ mod $\frakp$ has the same degree as $\phi$. We direct the reader to \cite[p. 58]{silverman:2007} for a formal definition, but we note that the map $f_c$ has good reduction at $\frakp$ if and only if $c$ has nonnegative $\frakp$-adic valuation. In particular, $f_c \in \QQbar[z]$ has everywhere good reduction if and only if $c \in \ZZbar$.

\begin{lem}[{\cite[Cor. B]{morton/silverman:1994}}]\label{lem:red}
Let $K$ be a number field, let $\phi(z) \in K(z)$ have degree at least two, and let $\frakp$ and $\frakq$ be primes in $K$ of good reduction for $\phi$ with different residue characteristics. Suppose $\alpha \in K$ has period $N$. Then
	\[
		N \le (N_{K/\bbQ}\frakp^2 - 1)(N_{K/\bbQ}\frakq^2 - 1),
	\]
where $N_{K/\bbQ}$ denotes the ideal norm on the ring of integers $\OK$.
\end{lem}

\begin{proof}[Proof of Theorem~\ref{thm:realize}]
Let $G$ be a strongly admissible graph. Define positive integers
	\begin{align*}
	D(G) &:= \deg \pi_G;\\
	B(G) &:= (2^{2D(G)} - 1)(3^{2D(G)} - 1);
	\end{align*}
where $\pi_G : X_1(G) \to \bbP^1$ is the projection onto the $t$-coordinate. Let $\{G_1,G_2,\ldots,G_n\}$ be the set of all admissible graphs which may be obtained from $G$ either by appending two preimages to an initial point of $G$, or by adding a new point of period at most $B(G)$ (and therefore also the subgraph generated by that point). Note that if $\tilde{G}$ is an admissible graph that properly contains $G$, then either $G_i \subseteq \tilde{G}$ for some $i = 1,\ldots,n$ or $\tilde{G}$ has a point of period greater than $B(G)$. For each $i \in \{1,\ldots,n\}$, denote by $D(G_i)$ the degree of $\pi_{G_i} : X_1(G_i) \to \bbP^1$.

Since the curves $X_1(G_i)$ are irreducible over $\bbQ$ (in fact, over $\bbC$) by Theorem~\ref{thm:main_irred}, it follows from Hilbert irreducibility that there exists a thin set $\calT(G) \subseteq \bbQ$ with the property that for all $c \in \bbQ \setminus \calT(G)$, the fibers of $\pi_{G_i}$ are inert over $c$ for each $i = 1,\ldots,n$. In other words, for all $c \in \bbQ \setminus \calT(G)$, the geometric preimages of $c$ under $\pi_{G_i}$ have degree $D(G_i)$ over $\bbQ$. Since there are infinitely many integers in the complement $\bbQ \setminus \calT$ for any thin set $\calT \subset \bbQ$ (see \cite[p. 26]{serre:2008}), we may fix an element $c \in \bbZ \setminus \calT(G)$ with $c \not \in \{-2,-1,0\}$; moreover, since $X_1(G)(\QQbar) \setminus U_1(G)(\QQbar)$ is a finite set, we may assume further that $\pi_G^{-1}(c) \subset U_1(G)(\QQbar)$.

Let $K$ be the field of definition of any preimage of $c$ under $\pi_G$. Then $[K:\bbQ] \le \deg \pi_G = D(G)$ and $G(f_c,K) \supseteq G$. It remains to show that $G(f_c,K)$ does not \emph{properly} contain $G$; it suffices, therefore, to show that $G(f_c,K)$ does not contain $G_i$ for any $i = 1,\ldots,n$ and that $f_c$ does not admit a $K$-rational point of period greater than $B(G)$.

Since $c \not \in \calT(G)$, the field of definition of any preimage of $c$ under $\pi_{G_i}$ has degree $D(G_i)$, which is strictly greater than $D(G)$ by Proposition~\ref{prop:graph_morphisms} and the fact that $G$ is strongly admissible. In particular, this implies that $c$ has no $K$-rational preimages under $\pi_{G_i}$; hence $G(f_c,K) \not \supseteq G_i$ for each $i = 1,\ldots,n$.

Since $c \in \bbZ$, the map $f_c$ has everywhere good reduction, so in particular $f_c$ has good reduction at any primes $\frakp_2$ and $\frakp_3$ in $\OK$ lying over $2$ and $3$, respectively. It then follows from Lemma~\ref{lem:red} that if $f_c$ has a $K$-rational point of period $N$, then
	\[
	N \le (N_{K/\bbQ} \frakp_2^2 - 1)(N_{K/\bbQ}\frakp_3^2 - 1) \le \left(\left(2^{D(G)}\right)^2 - 1\right)\left(\left(3^{D(G)}\right)^2 - 1\right) = B(G).
	\]
Thus $f_c$ has no $K$-rational points of period greater than $B(G)$. We conclude that $G(f_c,K) \cong G$.
\end{proof}

\section{A Galois representation associated to preperiodic points}\label{sec:galois}

We now return to the Galois theory associated to preperiodic points for quadratic polynomial maps. We begin by describing the Galois group of $L/\bbC(t)$, where $L$ is the extension of $\bbC(t)$ generated by all preperiodic points for $f_t$. Because $f_t$ is not PCF, and because $f_t$ has exactly $D(N)$ points of period $N$ for all $N \in \bbN$, the graph $G(f_t,\overline{\bbC(t)}) = G(f_t,L)$ is isomorphic to the direct limit
	\begin{equation}\label{eq:limit}
		\G := \varinjlim_{\substack{\text{admissible}\\\text{graphs } G}} G,
	\end{equation}
where the directed system of admissible graphs is ordered by inclusion. We therefore have a natural embedding
	\[
		\rho : \Gal(L/\bbC(t)) \hookrightarrow \Aut(\G)
	\]
obtained by restricting the action of $\sigma \in \Gal(L/\bbC(t))$ to the graph $G(f_t,L) \cong \G$. We will show that $\rho$ is actually an isomorphism. In order to do so, we require a lemma regarding Galois groups associated to those graphs which are normal in the sense of Definition~\ref{defn:normal}. For an admissible graph $G$, we will denote by $K_G/\bbC(t)$ the function field $\bbC(X_1(G))$.

\begin{lem}\label{lem:galois}
If $G$ is a normal graph, then $\Gal(K_G/\bbC(t)) \cong \Aut(G)$.
\end{lem}

\begin{proof}
Since $G$ is normal, we write
	\[
		G = \bigsqcup_{i=1}^n G_{M_i,N_i}
	\]
for some collection of integers $M_i \ge 0$ and $N_i \ge 1$, with the $N_i$ distinct. Note that because the cycle lengths of the graphs $G_{M_i,N_i}$ are distinct, we have
	\[
		\Aut(G) \cong \prod_{i=1}^n \Aut(G_{M_i,N_i}).
	\]
The proof of Theorem~\ref{thm:main_irred} shows that the function fields of $X_1(G_{M_i,N_i})$ are pairwise linearly disjoint over $\bbC(t)$, so
	\[
		\Gal(K_G/\bbC(t)) \cong \prod_{i=1}^n \Gal(K_{G_{M_i,N_i}}/\bbC(t)).
	\]
Finally, Corollary~\ref{cor:gal_closure} gives the isomorphism
	\[
		\Gal(K_{G_{M_i,N_i}}/\bbC(t)) \cong \Aut(G_{M_i,N_i}),
	\]
completing the proof.
\end{proof}

We now state our main Galois-theoretic result, which is a more precise statement of Theorem~\ref{thm:main_galois}.

\begin{thm}\label{thm:rep_iso}
Let $L$ be the extension of $\bbC(t)$ generated by all preperiodic points of $f_t$, and let $\G$ be the graph defined in \eqref{eq:limit}. Then
	\[
		\Gal(L/\bbC(t)) \cong \Aut(\G).
	\]
\end{thm}

\begin{proof}
We first note that since every admissible graph is contained in a normal graph, we also have
	\[
		\G = \varinjlim_{\substack{\text{normal}\\\text{graphs } G}} G.
	\]
If $G \hookrightarrow G'$ is an inclusion of normal graphs, then any automorphism of $G'$ restricts to an automorphism of $G$, since an automorphism of $G'$ must preserve preperiodic portraits. It follows that the groups $\Aut(G)$ with $G$ normal form an inverse system, with connecting maps given by restriction; moreover, the automorphism group of $\G$ is the inverse limit
	\begin{equation}\label{eq:invlim1}
		\Aut(\G) = \varprojlim_{\substack{\text{normal}\\\text{graphs } G}} \Aut(G).
	\end{equation}
On the other hand, we also have
	\[
		L = \varinjlim_{\substack{\text{admissible}\\\text{graphs } G}} K_G = \varinjlim_{\substack{\text{normal}\\\text{graphs } G}} K_G,
	\]
where $K_G = \bbC(X_1(G))$ as above; consequently,
	\begin{equation}\label{eq:invlim2}
		\Gal(L/\bbC(t)) = \varprojlim_{\substack{\text{normal}\\\text{graphs } G}} \Gal(K_G/\bbC(t)),
	\end{equation}
where again the connecting maps are given by restriction. From Lemma~\ref{lem:galois}, the groups $\Gal(K_G/\bbC(t))$ and $\Aut(G)$ are isomorphic, and in fact it is straightforward to see that for any inclusion $G \hookrightarrow G'$, the following diagram commutes:
\begin{center}
%    \begin{tikzpicture}[scale=4.5]
%\tikzset{vertex/.style = {}}
%\tikzset{edge/.style={decoration={markings,mark=at position 1 with %
%    {\arrow[scale=1.5,>=stealth]{>}}},postaction={decorate}}}
%\tikzset{loop/.style={min distance=10mm, in=-45, out=-135, decoration={markings,mark=at position 0.99 with %
%    {\arrow[scale=1.5,>=stealth]{>}}}, postaction={decorate}}}
%% vertices
%\node[vertex] (KG') at (0,.3) {$\Gal(K_{G'}/\bbC(t))$};
%\node[vertex] (KG) at (1,.3) {$\Gal(K_{G}/\bbC(t))$};
%\node[vertex] (G') at (0,0) {$\Aut(G')$};
%\node[vertex] (G) at (1,0) {$\Aut(G)$};
%%
%\draw[edge] (KG') to node[auto]{\scriptsize$\sigma \mapsto \left.\sigma\right|_{K_G}$} (KG);
%\draw[edge] (G') to node[auto]{\scriptsize$\tau \mapsto \left.\tau\right|_{G}$} (G);
%\draw[edge] (KG) to node[auto]{$\wr$} (G);
%\draw[edge] (KG') to node[auto]{$\wr$} (G');
%\end{tikzpicture}
	\includegraphics{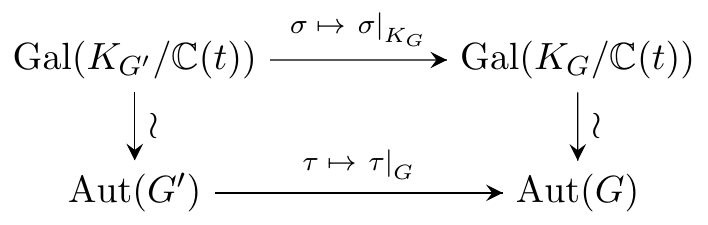}
\end{center}
%\[
%	\begin{tikzcd}[column sep=large]
%	\Gal(K_{G'}/\bbC(t)) \isoarrow{d} \arrow{r}{\sigma \mapsto \sigma|_{K_G}}
%		& \Gal(K_{G}/\bbC(t)) \isoarrow{d}\\
%	\Aut(G') \arrow{r}{\tau \mapsto \tau|_G}
%		& \Aut(G)
%	\end{tikzcd}.
%\]
Taking the inverse limits in \eqref{eq:invlim1} and \eqref{eq:invlim2} yields the result.
\end{proof}

\begin{rem}
Theorem~\ref{thm:rep_iso} is the extension to preperiodic points of a result due to Morton regarding periodic points. Morton showed in \cite[Thm. D]{morton:1998gal} that if $K \subset L$ is the extension of $\bbC(t)$ generated by adjoining all periodic points of $f_t$, then $\Gal(K/\bbC(t))$ consists of all permutations of the periodic points that commute with $f_t$. This is equivalent to saying that $\Gal(K/\bbC(t)) \cong \Aut(\H)$, where $\H \subset \G$ is the direct limit of all admissible graphs generated by periodic points, again ordered by inclusion; that is, $\H$ is the disjoint union
	\[
		\H = \bigsqcup_{N=1}^\infty G_{0,N}.
	\]
As noted immediately following Theorem~\ref{thm:bousch}, if we fix $N \ge 1$, then the group of allowable permutations of the points of period $N$ is isomorphic to a certain wreath product
	\[
		\bbZ/N\bbZ \wr S_{R(N)},
	\]
where $S_{R(N)}$ denotes the symmetric group on $R(N)$ letters; hence
	\begin{equation}\label{eq:wreath}
		\Gal(K/\bbC(t)) \cong \Aut(\H) \cong \prod_{N = 1}^\infty \left(\bbZ/N\bbZ \wr S_{R(N)}\right).
	\end{equation}
\end{rem}

Theorem~\ref{thm:rep_iso} actually fits into the more general framework of a natural Galois representation arising from preperiodic points for quadratic maps:

\begin{defn}
Let $K$ be a field, and let $f(z) \in K[z]$ be a quadratic polynomial. Let $\G_f = G(f,\Kbar)$ be the full directed graph of preperiodic points for $K$, and let
	\[
		\rho_{f,K} : \Gal(\Kbar/K) \to \Aut(\G_f)
	\]
be restriction to $\G_f$. We call $\rho_{f,K}$ the {\bf preperiodic Galois representation} over $K$ associated to $f$. In the case $f = f_c$ for some $c \in K$, we write $\G_c$ and $\rho_{c,K}$ for $\G_{f_c}$ and $\rho_{f_c,K}$, respectively.
\end{defn}

In characteristic zero, it is typically the case that $\G_f \cong \G$, so at the expense of excluding certain quadratic maps $f$, one could replace $\Aut(\G_f)$ with $\Aut(\G)$ so that the target of the representation does not depend on $f$. We note that Theorem~\ref{thm:rep_iso} is equivalent to saying the preperiodic Galois representation over $\bbC(t)$ associated to $f_t$ is surjective.

\begin{rem}\label{rem:arb}
The preperiodic Galois representation is related to another dynamically defined Galois representation called the {\it arboreal representation}. Let $\T$ denote the infinite complete rooted binary tree. If $f$ is a quadratic\footnote{The arboreal representation may be defined for rational maps of arbitrary degree $d \ge 2$, with the binary tree replaced by a $d$-ary tree, but we will restrict to $d = 2$ for the purposes of this discussion.} polynomial, and if $\alpha \in K$ is not in the orbit of the critical point of $f$, then there is a map $\rho_{f,\alpha,K} : \Gal(\Kbar/K) \to \Aut(\T)$ given by restriction to the set
	\[
		T(f,\alpha) = \{\beta \in \Kbar : f^m(\beta) = \alpha \text{ for some } m \ge 0\}
	\]
of iterated preimages of $\alpha$ under $f$, which may be given the structure of a directed graph isomorphic to $\T$. The map $\rho_{f,\alpha,K}$ is the arboreal representation over $K$ associated to $(f,\alpha)$. We recommend that the interested reader see \cite{jones:2013} for a nice survey on the topic.

The connection between these two representations lies in the fact that the automorphism group of the tree $\T$ appears in both. Suppose for the sake of exposition that $\G_f \cong \G$, so that the target of the preperiodic representation $\rho_{f,K}$ is $\Aut(\G)$. The group $\Aut(\G)$ has the structure of a wreath product $\Aut(\T) \wr \Aut(\H)$, where $\Aut(\H)$ is itself isomorphic to the direct product of wreath products in \eqref{eq:wreath}. One can see this by considering an automorphism $\sigma\in\Aut(\G)$ as acting in two steps: First, $\sigma$ acts by a graph automorphism on the tree $T(\f,P) \cong \T$ for {\it each} vertex $P$ of preperiod 1; the group of such automorphisms is
	\[
		\prod_{\substack{\text{$P$ a vertex}\\\text{of preperiod 1}}} \Aut(T_{\f,P}) \cong \Aut(\T)^\bbN.
	\]
Second, $\sigma$ acts by an automorphism of $\H$; since $\H$ contains both the periodic points and the points of preperiod 1, $\sigma$ also permutes the trees $T(\f,P)$ {\it as sets}. Though we do not provide the details of the proof, this conceptual description of $\Aut(\G)$ suggests that $\Aut(\G) \cong \Aut(\T)^\bbN \rtimes \Aut(\H)$, and this particular semidirect product is isomorphic to the wreath product $\Aut(\T) \wr \Aut(\H)$.
\end{rem}

As mentioned in the introduction, preperiodic points are a dynamical analogue of torsion points on elliptic curves, so one may ask whether the preperiodic representation exhibits properties known to be satisfied by the adelic Galois representation associated to elliptic curves. In particular, it would be interesting to know whether an analogue of Serre's open image theorem \cite[Thm. 3]{serre:1972} holds for this particular representation; i.e., we would like to know whether the image of $\rho_{f,K}$ in $\Aut(\mathbf{G}_f)$ should generally have finite index and, if so, to classify those exceptional quadratic maps $f(z) \in K[z]$ for which the index is infinite.

One obstruction to the image of $\rho_{f,K}$ having finite index is if $f$ is PCF. This is due to the fact that post-critical finiteness is a known obstruction (see \cite[Thm. 3.1]{jones:2013}) to the finite index problem for the arboreal representation mentioned in Remark~\ref{rem:arb}. We therefore pose the following question:

\begin{ques}\label{ques:index}
Let $K$ be a number field, let $f(z) \in K[z]$ be a quadratic polynomial, and assume $f$ is not PCF. Does the image of $\rho_{f,K}$ have finite index in $\Aut(\mathbf{G}_f)$?
\end{ques}

At present, this appears to be a difficult question. It would be interesting to have even one example of a parameter $c_0 \in \bbQ$ for which the image of $\rho_{c_0,\bbQ}$ is known to have finite index. Note that if we had a non-PCF parameter $c_0 \in \bbQ$ for which $\rho_{c_0,\bbQ}$ were {\it surjective} and for which $f_{c_0}$ had $D(N)$ points of period $N$ over $\QQbar$ for all $N \in \bbN$ (for example, if $c_0 \in \bbZ \setminus \{-2,-1,0\}$), then a much stronger version of Theorem~\ref{thm:realize} would hold: For each strongly admissible graph $G$, there would be a number field $K$ for which $G(f_{c_0},K) \cong G$.

As is often the case with questions of this type, the finite index problem becomes more tractable when we consider the function field setting. If $K$ is a function field over an algebraically closed field $F$, we say that a polynomial $f(z) \in K(z)$ is {\bf isotrivial} if $f$ is linearly conjugate to a polynomial defined over the field of constants $F$. We conclude with the following consequence of Theorem~\ref{thm:rep_iso}.

\begin{thm}
Let $K$ be a finitely generated extension of $\bbC$. For any non-isotrivial quadratic polynomial $f(z) \in K[z]$, the image of $\rho_{f,K}$ has finite index in $\Aut(\G_f) \cong \Aut(\G)$.
\end{thm}

\begin{proof}
By making a change of variables, we assume $f = f_u$ for some $u \in K$, and since $f$ is non-isotrivial, we further assume $u \notin \bbC$. Note that this implies that $K \supsetneq \bbC$ and $u$ is transcendental over $\bbC$. Let $L$ and $E = KL$ denote the fields obtained by adjoining all preperiodic points for $f_u$ to $\bbC(u)$ and $K$, respectively, so we have the following diagram:
\begin{center}
%\begin{tikzpicture}[scale=.85]
%\tikzset{vertex/.style = {}}
%\tikzset{edge/.style = {}}
%\node[vertex] (E) at  (0,2.5) {$E$};
%\node[vertex] (K) at  (0,1.25) {$K$};
%\node[vertex] (L) at  (1.5,1.875) {$L$};
%\node[vertex] (int) at  (1.5,.625) {$K \cap L$};
%\node[vertex] (Cu) at  (1.5,-.625) {$\bbC(u)$};
%%edges
%\draw[edge] (E) to (K);
%\draw[edge] (E) to (L);
%\draw[edge] (K) to (int);
%\draw[edge] (L) to (int);
%\draw[edge] (int) to (Cu);
%\end{tikzpicture}
	\includegraphics{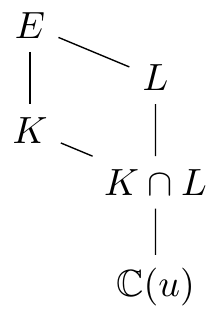}
\end{center}
Since $L/\bbC(u)$ is Galois, it follows from standard Galois theory (see \cite[Prop. 7.14]{milne:FGT53}) that the natural restriction map from $\Gal(E/K)$ to $\Gal(L/(K\cap L))$ is an isomorphism. Now, $K$ is a finitely generated extension of $\bbC(u)$ and $L$ is an algebraic extension of $\bbC(u)$, so $K \cap L$ is a finite extension of $\bbC(u)$. Thus $\Gal(L/(K\cap L))$ is a finite index subgroup of $\Gal(L/\bbC(u))$, and therefore the embedding
	\[
		\iota : \Gal(E/K) \overset{\sim}{\rightarrow} \Gal(L/(K \cap L)) \hookrightarrow \Gal(L/\bbC(u))
	\]
has image with finite index in its codomain. Since $\rho_{u,K} = \rho_{u,\bbC(u)} \circ \iota$, and since $\rho_{u,\bbC(u)}$ is an isomorphism by Theorem~\ref{thm:rep_iso}, we conclude that the image of $\rho_{u,K}$ has finite index in $\Aut(\G_u)$.
\end{proof}

\bibliography{C:/Dropbox/jdoyle}

\bibliographystyle{amsplain}

\end{document}